\documentclass[10pt]{amsart}
\usepackage{amsmath}
\usepackage{amssymb}
\usepackage{enumerate}
\usepackage{amsbsy}
\usepackage{amsfonts}
\usepackage{color}
\usepackage{upgreek}
\usepackage{xcolor}
\usepackage{tikz}

\newtheorem{thm}{Theorem}[section]
\newtheorem{lem}[thm]{Lemma}
\newtheorem{prop}[thm]{Proposition}

\numberwithin{equation}{section}

\begin{document}

	\title[Nonlinear Dirac equations on curved spacetime]{Global solutions to cubic Dirac and Dirac-Klein-Gordon systems on spacetimes close to the Minkowski space}

 \author[S. Hong]{Seokchang Hong}
    \address{Fakult\"at f\"ur Mathematik, Universit\"at Bielefeld, Bielefeld, Germany}
    \email{shong@math.uni-bielefeld.de}

	
	\begin{abstract}
		We establish global existence and derive sharp pointwise decay estimates of solutions to cubic Dirac and Dirac-Klein-Gordon systems on a curved background, close to the Minkowski spacetime. By squaring the Dirac operator, we reduce the analysis to a nonlinear wave-type equation involving spinorial connections, and apply energy estimates based on vector field methods and the hyperboloidal foliation framework, introduced by LeFloch–Ma. A key difficulty arises from the commutator structure of the Dirac operator, which exhibits significantly different behaviour from that of scalar field equations and requires refined control throughout the analysis, particularly due to the spacetime-dependent gamma matrices, which reduce to constant matrices in the flat Minkowski spacetime.
	\end{abstract}

		\maketitle

\section{Introduction}
\subsection{Motivation and Background}
In the early twentieth century, the Dirac equation was originally derived by factorising the Klein-Gordon equation $(\Box-m^2)\phi=0$. Dirac's approach \cite{D} was successful, and the resulting equation plays a central role in quantum field theory, effectively describing the spin $\frac12$-fermions such as electrons. On the Minkowski spacetime $(\mathbb R^{1+3},\mathbf m)$, where the metric $\mathbf m$ is given by $\textrm{diag}(-1,+1,+1,+1)$, the linear Dirac equation takes the form
\begin{align}
    -i\gamma^\mu\partial_\mu\psi+m\psi = 0,
\end{align}
where $m\ge0$ denotes the mass of the particle and $\psi:\mathbb R^{1+3}\to \mathbb C^4$ is a complex-valued spinor field. The gamma matrices $\gamma^\mu\in\mathbb C^{4\times4}$, $\mu=0,1,2,3$ satisfies the algebraic relation: $\gamma^\mu\gamma^\nu+\gamma^\nu\gamma^\mu=-2m^{\mu\nu}I_{4}$, where $I_4$ is the $4\times4$ identity matrix.

From the mathematical perspective, the Dirac equation is a first-order hyperbolic system. Hence several analytic techniques, which have been well-developed in the wave and the Klein-Gordon equations, can be applied to the analytic study of the Dirac equation, in that the Dirac equation can be formulated as the half-wave in the massless case $m=0$ and the half-Klein-Gordon equations in the massive case $m>0$, respectively. 

In recent decades, nonlinear problems of the Dirac equations have been extensively studied, especially in the context of dispersive equations. Global well-posedness and scattering results for cubic Dirac equations have been established by \cite{machihara,behe,boucan}. The Dirac-Klein-Gordon system has been well-studied by \cite{anfoselb,wang,beherr,canhe}.

When one is concerned with the Dirac equation in a general Lorentz manifold $(\mathcal M,g)$, the usual partial differentiation in the equation must then be replaced by the spinorial covariant derivative. To be precise, the Dirac equation on a curved spacetime is given by
\begin{align*}
    -i\gamma^\mu(t,x)\mathbf D_\mu\psi +m\psi =0,
\end{align*}
where $\mathbf D_\mu$ is the covariant derivative acting on the spinor field $\psi$. The precise definition of $\mathbf D_\mu$ will be discussed in Section \ref{sec:dirac-operator}. In contrast to the flat case, the gamma matrices now depend on spacetime points. 
 Although the covariant Dirac equation naturally arises from the spin structure of a given Lorentz manifold, this is also essential for quantum field theory in curved spacetime, including physical phenomena such as Hawking radiation, and semiclassical gravity.

Recently, the mathematical analysis of the Dirac equation on a curved spacetime has also seen notable progress. B\"ar \cite{bar} investigated spectral properties of the covariant Dirac operator; Cacciafesta, Suzzoni, Meng \cite{caccia1} derived Strichartz estimates. More recently, the sharp decay for the linear massless Dirac equation on the Schwarzschild background was established by Ma and Zhang \cite{mazhang}.

Nevertheless, most of the previous results for the Dirac equation on a curved background have focused on the linear and massless setting. Moreover, the nonlinear analysis of the Dirac equation on a curved spacetime has remained relatively less explored.

Motivated by these considerations, we investigate the global behaviour of nonlinear Dirac equations on an asymptotically flat background, which is close to the Minkowski spacetime. We then establish global existence and sharp pointwise decay estimates of solutions to the systems in this setting.

\subsection{Main results}
The cubic Dirac equations and the Dirac-Klein-Gordon (DKG) system on a curved background are given by
\begin{align}
\begin{aligned}
	(-i\gamma^\mu\mathbf D_\mu+M)\psi = (\psi^\dagger\gamma^0\psi)\psi, \\
	\psi|_{t=2} := \psi_0,
	\end{aligned}
\end{align}
and
\begin{align}
\begin{aligned}
	(-i\gamma^\mu\mathbf D_\mu+M)\psi &= \phi\psi, \\
	(\Box_g-m^2)\phi &= \psi^\dagger\gamma^0\psi, \\
	(\psi,\phi,\partial_t\phi)|_{t=2}&:= (\psi_0,\phi_0,\phi_1),
	\end{aligned}
\end{align}
where $\psi:\mathbb R^{1+3}\to\mathbb C^4$ is a spinor field, $\phi:\mathbb R^{1+3}\to\mathbb C$ is a scalar field, and $\psi^\dagger$ is the complex conjugate transpose of $\psi$. Our aim is to establish global existence and sharp pointwise decay estimates of solutions to the systems on an asymptotically flat background, close to the Minkowski spacetime. To be precise,
we consider a smooth Lorentzian metric in $\mathbb R^+_t\times\mathbb R^3_x$, such that
\begin{enumerate}
	\item $t=const.$ is space-like.
	\item The vector field $\partial_t$ is a (time-like) Killing field, i.e., $g$ is stationary.
	\item $g$ is asymptotically flat in the following sense:
	\begin{align*}
		g= m +g_{sr}+g_{lr},
	\end{align*}
	where $g_{lr}$ is a long range spherically symmetric component and $g_{sr}$ is a short range component such that
    \begin{align}\label{metric-decay}
    \begin{aligned}
        |\partial_x^\alpha g_{lr}|\le c^{lr}_\alpha \langle r\rangle^{-1-|\alpha|}, \quad |\partial_x^\alpha g_{sr}| \le c^{sr}_\alpha \langle r\rangle^{-2-|\alpha|}, \quad |\alpha|\le N+2, 
        \end{aligned}
    \end{align}
    where $N$ will be fixed later.
    \item We suppose that the constants $c_\alpha^{lr}$ and $c_{\alpha}^{sr}$ are sufficiently small, for every $|\alpha|\le N+2$, which ensures that the curved geometry $(\mathbb R^+_t\times\mathbb R^3_x,g)$ is close to the Minkowski spacetime $(\mathbb R^+_t\times\mathbb R^3_x,\mathbf m)$.
\end{enumerate}
We also write
\begin{align*}
	g_{lr} = g_{lr,tt}(r)dt^2+g_{lr,rr}(r)dr^2+g_{lr,\omega\omega}(r)r^2d\omega^2, \\
	g_{sr} = g_{sr,tt}(x)dt^2+2g_{sr,tj}(x)dtdx^j+g_{sr,ij}(x)dx^idx^j.
\end{align*}
Now we use the normalized coordinates introduced in \cite{metato,tataru} and write
\begin{align*}
	\Box_g = \Box+g^\omega_{lr}(r)\Delta_\omega+\partial_\alpha g^{\alpha\beta}_{sr}\partial_\beta 
\end{align*}
where $g^\omega_{lr}(r)\approx \langle r\rangle^{-3}$ and $g^{00}_{sr}=0$. In other words, the metric $g$ can be written as
\begin{align*}
	g &= m+g_{lr}+g_{sr} \\
	& = m+g_{\omega}(r)r^2d\omega^2+2g_{sr,tj}(x)dtdx^j+g_{sr,ij}(x)dx^idx^j,
\end{align*}
where $m$ is the Minkowski metric.
In other words, we consider the metric $g=-dt^2+2g_{0j}(x)dtdx^j+g_{ij}(x)dx^idx^j$ where the corresponding Laplace-Beltrami operator is given by
\begin{align*}
	\Box_g = \Box+g^\omega_{lr}(r)\Delta_\omega+\partial_\alpha g^{\alpha\beta}_{sr}\partial_\beta .
\end{align*} 
We refer the readers to Section 2 of \cite{tataru} for the details of the normalised coordinates.

In what follows, we restrict ourselves to sufficiently smooth functions which are spatially compactly supported inside the forward light cone $\{(t,x)\in \mathbb R^{+}_t\times\mathbb R^3_x: |x|< t-1\}$ with $t\ge2$. Then we foliate this region via the hyperboloid $\Sigma_{\tau}:=\{(t,x): \tau = \sqrt{t^2-|x|^2}\}$.
Define the energy of the spinor $\psi:\mathbb R^{1+3}\to\mathbb C^4$,
\begin{align*}
	E_c[\psi](\tau) := \int_{\Sigma_\tau} \sum_{j=1}^3\left| \partial_j\psi+\frac{x^j}{t}\partial_t\psi \right|^2+\left| \frac\tau t\partial_t\psi\right|^2+c^2|\psi|^2\,dx,
\end{align*}
where $|\psi|^2 = \psi^\dagger\psi$ and $\psi^\dagger$ is the complex conjugate transpose of $\psi$. The energy of the scalar field $\phi$ from the (DKG) system is defined in the obvious way.
We refer the readers to the Appendix for the detailed discussion of the derivation of the energy.
We denote by $\partial^I$ and $L^J$ the product of $|I|$-partial derivatives $\partial_t,\partial_i$, and the product of $|J|$ Lorentz boost $L_i=t\partial_i+x^i\partial_t$. Then we consider the initial value problems for the cubic Dirac and the (DKG) systems:
\begin{align}\label{eq-cubic-dirac}
\begin{aligned}
	(-i\gamma^\mu\mathbf D_\mu+M)\psi = (\psi^\dagger\gamma^0\psi)\psi, \\
	\psi|_{t=2} := \psi_0\in H^N(\mathbb R^3),
	\end{aligned}
\end{align}
and
\begin{align}\label{eq-dkg}
\begin{aligned}
	(-i\gamma^\mu\mathbf D_\mu+M)\psi &= \phi\psi, \\
	(\Box_g-m^2)\phi &= \psi^\dagger\gamma^0\psi, \\
	(\psi,\phi,\partial_t\phi)|_{t=2}&:= (\psi_0,\phi_0,\phi_1)\in H^N\times H^N\times H^{N-1}(\mathbb R^3),
	\end{aligned}
\end{align}
where $N\ge13$ is a fixed integer. We let $\epsilon_0=\epsilon_0(N,g)$ be a sufficiently small quantity. By $\epsilon_0=\epsilon_0(g)$ we mean that the $\epsilon_0$ is dependent on the constant $c_\alpha$ in \eqref{metric-decay}. Assume that the initial data  $\psi_0,\phi_0,\phi_1$ satisfy the following smallness condition:
\begin{align}
\|\psi_0\|_{H^N(\mathbb R^3)}+\|\phi_0\|_{H^N(\mathbb R^3)}+\|\phi_1\|_{H^{N-1}(\mathbb R^3)} \le \varepsilon,	
\end{align}
where $0<\varepsilon<\epsilon_0$. By Theorem 11.2.1 of \cite{flochma1}, one can construct a local-in-time solution from the data given on the initial hyperboloid $\Sigma_{\tau_0}$ with $\tau_0=2$, which guarantees that\footnote{One should notice that the aforementioned theorem in \cite{flochma1} concerns local solutions to the Klein-Gordon equations, while we are concerned with the Dirac equation, which is first order. Indeed, we will work on nonlinear wave-type equations by squaring the Dirac operator. However, one cannot apply the results from \cite{flochma1} directly, since the nonlinear equations involve lower-order terms which are not covered by Theorem 11.2.1 of \cite{flochma1}. Thanks to the closeness to the Minkowski space, this gap is not harmful, since the lower-order terms can be absorbed somewhere. } on $\Sigma_{\tau_0}$ and for all $|I|+|J|\le N$,
\begin{align}
E_M[\partial^IL^J\psi](\tau_0)^\frac12 \le C_0\varepsilon, \quad 	E_m[\partial^IL^J\phi](\tau_0)^\frac12 \le C_0\varepsilon,
\end{align}
where $C_0$ is an absolute constant. Now we state our main results.
\begin{thm}\label{thm-global-cubic}
	Let $N\ge11$ be a fixed number and let the metric $g$ be given as \eqref{metric-decay}. There exists a small quantity $\epsilon_0=\epsilon_0(N,g)>0$ such that for all $0<\varepsilon<\epsilon_0$ and for all spatially compactly supported initial data $\psi_0$ satisfying the smallness condition $\|\psi_0\|_{H^N(\mathbb R^3)}\le\varepsilon$, the Cauchy problem of the cubic Dirac equation \eqref{eq-cubic-dirac} admits a global-in-time solution $\psi$ and the solution $\psi=\psi(t,x)$ satisfy the following energy bound: for all $\tau\ge\tau_0$
	\begin{align}
	E_M[\partial^IL^J\psi](\tau) \lesssim \varepsilon^2
	\end{align}
	and the following pointwise decay estimate:
	\begin{align}
	\sup_{(t,x)\in \Sigma_\tau}t^\frac32|\psi(t,x)| \lesssim \varepsilon .
	\end{align}
	\end{thm}
We note that the global well-posedness and scattering for the cubic Dirac equation on a weakly asymptotically flat background has already been established by the present author and Herr \cite{herrh}, where we proved endpoint Strichartz estimates for the half-Klein–Gordon equation and obtained global results for small initial data in $H^s(\mathbb R^3)$ with $s > 1$, covering the whole subcritical regime.
In light of that result, Theorem \ref{thm-global-cubic} seems a special case, which is included in \cite{herrh}.

However, our main novelty of this paper lies in global result on the (DKG) system, which possesses quadratic nonlinearity.

\begin{thm}\label{thm-global-dkg}
	Let $N\ge11$ be a fixed number and let the metric $g$ be given as \eqref{metric-decay}. There exists a small quantity $\epsilon_0=\epsilon_0(N,g)>0$ such that for all $0<\varepsilon<\epsilon_0$ and for all compactly supported initial data $(\psi_0,\phi_0,\phi_1)$ satisfying the smallness condition $\|\psi_0\|_{H^N(\mathbb R^3)}+\|\phi_0\|_{H^N}+\|\phi_1\|_{H^{N-1}}\le\varepsilon$, the Cauchy problem of the (DKG) system \eqref{eq-dkg} admits global-in-time solutions $(\psi,\phi)$ and the solutions $\psi=\psi(t,x)$, $\phi=\phi(t,x)$ satisfy the following energy estimates: for all $\tau\ge\tau_0$
	\begin{align}
	\begin{aligned}
	E_M[\partial^IL^J\psi](\tau)+E_m[\partial^IL^J\phi](\tau)\lesssim\varepsilon^2
	\end{aligned}
	\end{align}
	and the pointwise decay
	\begin{align*}
		\sup_{(t,x)\in \Sigma_\tau}t^\frac32|\psi(t,x)|+\sup_{(t,x)\in \Sigma_\tau}t^\frac32|\phi(t,x)| \lesssim \varepsilon .
	\end{align*}
		\end{thm}
Our results, Theorem \ref{thm-global-cubic} and Theorem \ref{thm-global-dkg}, appear to be one of the first steps toward a quantitative understanding of nonlinear Dirac systems on curved backgrounds. Our approach draws inspiration from techniques developed for scalar equations, particularly the vector field approach and hyperboloidal method, introduced by Klainerman \cite{klai1} and further refined in the works of LeFloch - Ma \cite{flochma1}.
While our setting is distinct, we follow the analytic spirit of these earlier developments.
        
The proof mainly relies on the vector field approach. We refer the readers to \cite{donglefloch,dongli,donglimayuan,qzhang,qzhang1} for the works concerning the Dirac equations in the flat setting via vector field methods.

\subsection{Technical challenge on a curved spacetime}
We remark that an extension of existing results concerning nonlinear Dirac equations on the Minkowski spacetime to a curved background introduces several new technical difficulties.
\subsubsection*{ $\bullet$ Sensitivity to geometric structure} First, the Dirac field is much more sensitive to the geometry of the spacetime. Indeed, the curved geometry $g$ determines the algebraic relation:
\begin{align*}
	\frac12(\gamma^\mu\gamma^\nu+\gamma^\nu\gamma^\mu) = -g^{\mu\nu}I_4,
\end{align*}
which the gamma matrices must obey. This algebraic structure then determines how the Dirac operator is defined: $-i\gamma^\mu\mathbf D_\mu\psi+m\psi = 0$. This is evident upon squaring the Dirac equation, which yields the following second-order equation:\footnote{By definition of the spinorial connection $\Gamma_\mu$, it transforms as a $1$-form under the change of coordinates. See also the proof of Proposition \ref{prop-squaring-dirac} for the details. Thus, all the lower-order terms behave as scalar functions.}
\begin{align*}
	\left(\Box_g-m^2+2\Gamma^\mu\partial_\mu+\nabla^\mu\Gamma_\mu+\Gamma^\mu\Gamma_\mu+\frac{\mathbf R}{4}\right)\psi =0,
\end{align*}
where $\mathbf R$ is the scalar curvature associated to the metric $g$, $\Gamma_\mu$ is the spinorial connection, and $\nabla_\mu$ is the covariant derivative acting on a scalar field. See also Proposition \ref{prop-squaring-dirac}. The second-order equation derived from the Dirac field possesses a purely geometric quantity and also a spinorial potential term, which does not appear in scalar field equations, such as the wave or the Klein-Gordon. Moreover, the presence of such lower-order terms makes global-in-time problem more subtle and requires delicate analysis, unlike the scalar field equations. One way of remedying this difficulty is to impose a stronger assumption on the metric. In fact, by an asymptotically flat background which is close to the Minkowski spacetime, we mean the metric $g$ satisfies the inequality of the form  $|\partial_x^\alpha(g-m)|\le c_\alpha\langle r\rangle^{-1-|\alpha|}$, where the constants $c_\alpha$ are sufficiently small, so that trapping is excluded. However, due to the presence of the lower-order terms, this smallness might not be enough, and one needs even smaller constants $c_\alpha$, to close the bootstrap argument and establish global solutions. Roughly speaking, the Dirac field demands much stronger assumptions on the underlying geometry, in order to achieve an analogue result as in the wave or the Klein-Gordon equations.
\subsubsection*{ $\bullet$ Effects of the mass term on decay and conformal structure}
Second, the presence of the mass term $m>0$ introduces additional difficulties on a curved background, unlike on the Minkowski spacetime, where the mass term improves the time-decay compared to the wave equation. For the massive Dirac, in which case one does not  have conformal invariance, the vector field such as $S=t\partial_t+r\partial_r$ is not useful anymore. Even worse, the mass term interacts unfavourably with the geometry and makes the time decay to be $t^{-\frac56}$, much weaker than $t^{-\frac32}$, especially in the presence of a trapping region. See \cite{finster} and \cite{pasqua}.

In this work, under the non-trapping assumption we recover the optimal decay $t^{-\frac32}$, thereby compensating the absence of conformal symmetry. 
This suggests that extending our analysis to systems such as the Maxwell-Dirac, where the gauge field decays only as $t^{-1}$, introduces an additional challenge. 
We expect that extending the methods developed in this paper to such a system would be an interesting direction for future research.

\subsection{Main strategy} Now we give a key to the proof.
\subsubsection*{ $\bullet$ Squaring the Dirac operator}
We first square the Dirac operator. For the inhomogeneous Dirac equation $(-i\gamma^\mu\mathbf D_\mu+M)\psi = F$, by acting the Dirac operator $\gamma^\nu\mathbf D_\nu$ on both sides, we obtain
\begin{align*}
	\left(\Box_g-M^2+2\Gamma^\mu\partial_\mu+\nabla^\mu\Gamma_\mu+\Gamma^\mu\Gamma_\mu+\frac{\mathbf R}{4}\right)\psi = -i\gamma^\mu\mathbf D_\mu F-MF,
\end{align*}
where $\Box_g$ is the usual Laplace-Beltrami operator acting on a scalar field associated to the metric $g$. See also Proposition \ref{prop-squaring-dirac}. Due to the spinorial connection terms, one cannot view this equation as the usual Klein-Gordon equation. However, using an obvious inequality $|\Gamma_\mu|\lesssim |\partial g|$, we can consider this wave-type equation as the Klein-Gordon equation with a potential:
\begin{align*}
	(\Box_g-M^2)\psi = V\psi-2\Gamma^\mu\partial_\mu\psi-i\gamma^\mu\mathbf D_\mu F-MF,
\end{align*}
where the poitential $V$ satisfies the inequality $|V(x)|\lesssim |\partial^2g|+|\partial g|^2$. In what follows, the cubic Dirac and the (DKG) systems are transferred to the following nonlinear wave-type equations:
\begin{align*}
	(\Box_g-M^2)\psi = V\psi-2\Gamma^\mu\partial_\mu\psi-i\gamma^\mu\mathbf D_\mu (\psi^\dagger\gamma^0\psi)\psi -M(\psi^\dagger\gamma^0\psi)\psi,
\end{align*} 
and
\begin{align*}
	(\Box_g-M^2)\psi &= V\psi-2\Gamma^\mu\partial_\mu\psi-i\gamma^\mu\mathbf D_\mu \phi\psi-M\phi\psi, \\
	(\Box_g-m^2)\phi & = \psi^\dagger\gamma^0\psi,
\end{align*}
respectively.
A notable distinction is that the Dirac equation exhibits a purely geometric quantity, such as the scalar curvature, unlike scalar field equations such as wave or Klein-Gordon. This is one of the reflections that the Dirac equation is strongly tied to the geometry of the spacetime. Due to these lower-order terms $V\psi$ and $\Gamma^\mu\partial_\mu\psi$, global analysis is more delicate compared to the analysis on scalar fields.
\subsubsection*{ $\bullet$ Hyperboloidal foliation}
Recently, the foliation via the hyperboloid of the interior of the light cone $\{t=|x|\}$ turns out to be effective in the analysis of scalar field equations \cite{flochma1}. Indeed, the hyperboloid is invariant under the Lorentz transformation on the Minkowski spacetime. As we are concerned with an asymptotically flat spacetime with the non-trapping condition, this advantage remains to be held, and we will work with the energy naturally defined on the hyperboloid foliation and exploit the Klainerman-Sobolev inequality on the hyperboloid. We refer to \cite{flochma,flochma2,flochma3,qwang} for the hyperboloidal approach.
\begin{figure}[ht]\label{fig-hyper}
\centering
\begin{tikzpicture}[scale=1.0]
  \draw[->] (-5,0) -- (5,0) node[right] {$x$};
  \draw[->] (0,0) -- (0,5) node[above] {$t$};


  \draw[thick,gray] (0,0) -- (4.5,4.5);
  \draw[thick,gray] (0,0) -- (-4.5,4.5);
  \node at (3.7,3.1) {\footnotesize $t=|x|$};

  \draw[dashed,red,thick] (-2,2) -- (2,2) ;
  \node at (1.8,1.8) {\textcolor{red}{\footnotesize $t = 2 $}};

  \draw[dashed,thick] (0,1) -- (4,5);
  \draw[dashed,thick] (0,1) -- (-4,5);
  \node at (3.9,5) {\textcolor{black}{\footnotesize $t = |x| + 1$}};

  \foreach \t in {2,3,4} {
    \draw[domain={-\t}:{\t},smooth,variable=\x,blue,thick]
      plot ({\x},{sqrt(\x*\x+\t*\t)});
  }
    \node at (2.8,4.6) {\footnotesize $\Sigma_{\tau=4}$};
  
  \node at (2.5,3.6) {\footnotesize $\Sigma_{\tau=3}$};

   \node at (2.2,2.6) {\footnotesize $\Sigma_{\tau=2}$};


\end{tikzpicture}
\caption{Hyperboloidal foliation and the region where $|x| < t - 1$.}
\end{figure}
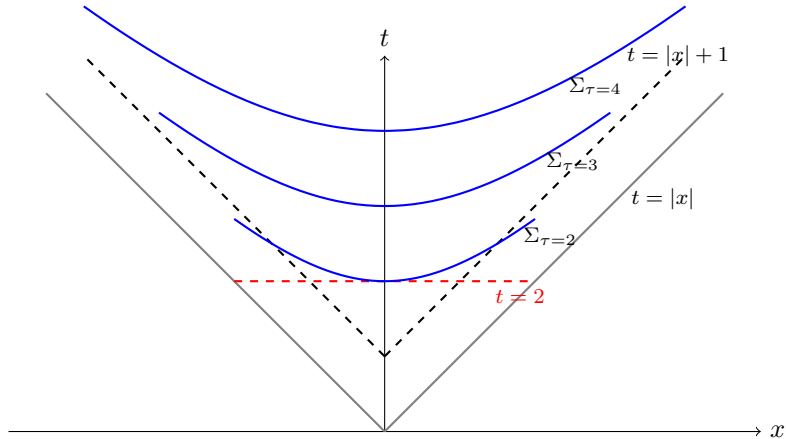

To be precise, we consider smooth functions spatially supported in the regime $\{ (t,x)\in\mathbb R^{+}_t\times\mathbb R^3_x :|x|\le t-1\}$. The dashed black lines represent the straight lines $t = |x| + 1$, which bound the support region from below. We further restrict ourselves to $t\ge2$. Then each blue curve $\Sigma_\tau = \{(t,x): t^2 - |x|^2 = \tau^2\}$ corresponds to a hyperboloidal time slice, providing a foliation of the interior of the light cone suitable. We define the energy on the hyperboloid $\Sigma_\tau$ as
\begin{align*}
    E_c[\psi](\tau) := \int_{\Sigma_\tau} \sum_{j=1}^3\left| \partial_j\psi+\frac{x^j}{t}\partial_t\psi \right|^2+\left| \frac\tau t\partial_t\psi\right|^2+c^2|\psi|^2\,dx.
\end{align*}
Now we transfer the Cauchy problems of the system in terms of the initial hypersurface $\{t=2\}$ to the Cauchy problems in terms of the initial hyperbolic surface $\{\tau=2\}$.
We first note that the support condition $|x|<t-1$ implies $2\le t\le \frac52$ on the hyperboloid $\Sigma_2$. 
Given a sufficiently small quantity $\epsilon_0=\epsilon_0(N,g)$, we consider the initial data $\psi_0,\phi_0,\phi_1$ at $t=2$ satisfying $\|\psi_0\|_{H^N}+\|\phi_0\|_{H^N}+\|\phi_1\|_{H^{N-1}}\le\varepsilon$, with $0<\varepsilon<\epsilon_0$. Then it suffices to construct a local-in-time solution to the system with the existence time $\frac12$. Indeed, this can be achieved by the discussion in Section 11 of \cite{flochma1}. Furthermore, we also have
\begin{align*}
    E_M[\partial^IL^J\psi](\tau_0)^\frac12 \le C_0\varepsilon, \quad 	E_m[\partial^IL^J\phi](\tau_0)^\frac12 \le C_0\varepsilon,
\end{align*}
for all $|I|+|J|\le N$ and some absolute constant $C_0>1$. In consequence, we have transferred the original problems to the Cauchy problems in terms of hyperbolic time.

\subsubsection*{ $\bullet$ Commutators with $\gamma^\mu\mathbf D_\mu$}
After squaring the Dirac operator, one cannot completely transfer the problem of nonlinear Dirac equations to the problem of nonlinear Klein-Gordon equation, due to the inhomogeneous term $\gamma^\mu\mathbf D_\mu F$, as well as the lower order terms. This forces us to consider the commutators of several vector fields with the Dirac operator. Remarkably, the spinor structure demands one to take into account the spinor-adapted commutator, which does not appear in the analysis of scalar fields. Indeed, the commutator $[\gamma^\mu\partial_\mu,L_i]$ with the Lorentz boost $L_i$ does not vanish even in the flat spacetime, and the Lorentz boost must be modified as $\widehat{L_i}=L_i+\frac12\gamma_0\gamma_i$ and one easily sees $[\gamma^\mu\partial_\mu,\widehat{L_i}]=0$ in the flat setting. It is obvious that $[\gamma^\mu\mathbf D_\mu, \widehat{L_i}]\neq0$ on a curved background, which requires a delicate analysis. Nevertheless, we observe that all the commutators turn out to be acceptable error terms, and estimates of the nonlinear terms $\gamma^\mu\mathbf D_\mu F$ will be reduced to estimates of scalar fields. See Proposition \ref{prop-dirac-reduce}.
\subsubsection*{ $\bullet$ The Bootstrap argument}
We fix an integer $N\ge11$. Let the initial data $\psi_0,\phi_0,\phi_1$ with the smallness condition: $\|\psi_0\|_{H^N(\mathbb R^3)}+\|\phi_0\|_{H^N(\mathbb R^3)}+\|\phi_1\|_{H^{N-1}(\mathbb R^3)} \le \varepsilon$ be given, with $0<\varepsilon<\epsilon_0$. The discussion of Section 11 of \cite{flochma1} implies that one can establish local well-posedness for the aforementioned nonlinear wave-type equation derived by squaring the Dirac operator, and the time of existence $T=T(\varepsilon,g)$ of the local solution tends to infinity as the size of the initial data approaches zero. One has to note that the time $T(\varepsilon,g)$ depends on the metric $g$ as well as the size of the data, due to the presence of the lower order terms $V\psi$ and $\Gamma^\mu\partial_\mu\psi$. This local well-posedness result ensures that on the initial hyperboloid $\Sigma_{\tau_0}$ with $\tau_0=2$, for all $|I|+|J|\le N$,
\begin{align*}
	E_M[\partial^IL^J\psi](\tau_0)^\frac12 \le C_0\varepsilon, \quad 	E_m[\partial^IL^J\phi](\tau_0)^\frac12 \le C_0\varepsilon,
\end{align*}
with an absolute constant $C_0$. Then we assume that on some hyperbolic time interval $[\tau_0,\tau_1]$, the following energy bounds hold for some constants $C,\delta,\varepsilon$:
\begin{align*}
	E_M[\partial^IL^J\psi](\tau)^\frac12 \le C\varepsilon \tau^{\frac12+(|I'|+|J|)\delta}, \quad N-3\le |I|+|J|\le N, \\
	E_M[\partial^IL^J\psi](\tau)^\frac12 \le C\varepsilon \tau^{(|I'|+|J|)\delta}, \quad  |I|+|J|\le N-4, \\
	E_m[\partial^IL^J\phi](\tau)^\frac12 \le C\varepsilon \tau^{\frac12+(|I'|+|J|)\delta}, \quad N-3\le |I|+|J|\le N, \\
	E_m[\partial^IL^J\phi](\tau)^\frac12 \le C\varepsilon \tau^{(|I'|+|J|)\delta}, \quad  |I|+|J|\le N-4,
\end{align*}
where we put $\partial^I=\partial_t^{I'}\partial_x^{I''}$. Note that the energy growth depends on the number of derivatives $\partial_t$ and the Lorentz boost $L$. Then, the main goal of the bootstrap argument is to improve the aforementioned assumptions as follows: for sufficiently small $\varepsilon$ and some $C>4C_0$ and $\frac1{10N}\le\delta\le \frac1{5N}$,
\begin{align*}
	E_M[\partial^IL^J\psi](\tau)^\frac12 \le \frac12C\varepsilon \tau^{\frac12+(|I'|+|J|)\delta}, \quad N-3\le |I|+|J|\le N, \\
	E_M[\partial^IL^J\psi](\tau)^\frac12 \le \frac12C\varepsilon \tau^{(|I'|+|J|)\delta}, \quad  |I|+|J|\le N-4, \\
	E_m[\partial^IL^J\phi](\tau)^\frac12 \le \frac12C\varepsilon \tau^{\frac12+(|I'|+|J|)\delta}, \quad N-3\le |I|+|J|\le N, \\
	E_m[\partial^IL^J\phi](\tau)^\frac12 \le \frac12C\varepsilon \tau^{(|I'|+|J|)\delta}, \quad  |I|+|J|\le N-4,
\end{align*}
which implies that $\tau_1=+\infty$ and we have global solutions $\psi$ and $(\psi,\phi)$ to the systems. To close the bootstrap argument, by energy inequality Proposition \ref{energy-ineq}, we have to show that
\begin{align}\label{bootest-1}
    \begin{aligned}
        \int_{\tau_0}^\tau \|[\Box_g,\partial^IL^J]\phi\|_{L^2(\Sigma_{\tau'})}\,d\tau' \le \frac12 C\varepsilon \tau^{\frac12+(|I'|+|J|)\delta}, \quad N-3\le |I|+|J|\le N, \\
          \int_{\tau_0}^\tau \|[\Box_g,\partial^IL^J]\phi\|_{L^2(\Sigma_{\tau'})}\,d\tau' \le \frac12 C\varepsilon \tau^{(|I'|+|J|)\delta}, \quad  |I|+|J|\le N-4, \\
            \int_{\tau_0}^\tau \|\partial^IL^J\Gamma^\mu\partial_\mu\psi\|_{L^2(\Sigma_{\tau'})}\,d\tau' \le \frac12 C\varepsilon \tau^{\frac12+(|I'|+|J|)\delta}, \quad N-3\le |I|+|J|\le N, \\
             \int_{\tau_0}^\tau \|\partial^IL^J\Gamma^\mu\partial_\mu\psi\|_{L^2(\Sigma_{\tau'})}\,d\tau' \le \frac12 C\varepsilon \tau^{(|I'|+|J|)\delta}, \quad |I|+|J|\le N-4, \\
 \int_{\tau_0}^\tau \|\partial^IL^JV\psi\|_{L^2(\Sigma_{\tau'})}\,d\tau' \le \frac12 C\varepsilon \tau^{\frac12+(|I'|+|J|)\delta}, \quad N-3\le |I|+|J|\le N, \\
  \int_{\tau_0}^\tau \|\partial^IL^JV\psi\|_{L^2(\Sigma_{\tau'})}\,d\tau' \le \frac12 C\varepsilon \tau^{(|I'|+|J|)\delta}, \quad  |I|+|J|\le N-4,
    \end{aligned}
\end{align}
and for the cubic problem we prove
\begin{align}\label{bootest-2}
    \begin{aligned}
          \int_{\tau_0}^\tau \|\partial^IL^J\gamma^\mu\mathbf D_\mu[(\psi^\dagger\gamma^0\psi)\psi]\|_{L^2(\Sigma_{\tau'})}\,d\tau' \le \frac12 C\varepsilon \tau^{\frac12+(|I'|+|J|)\delta}, \quad N-3\le |I|+|J|\le N, \\
    \int_{\tau_0}^\tau \|\partial^IL^J\gamma^\mu\mathbf D_\mu[(\psi^\dagger\gamma^0\psi)\psi]\|_{L^2(\Sigma_{\tau'})}\,d\tau' \le \frac12 C\varepsilon \tau^{(|I'|+|J|)\delta}, \quad  |I|+|J|\le N-4,
    \end{aligned}
\end{align}
and for the (DKG) system we prove
\begin{align}\label{bootest-3}
    \begin{aligned}
          \int_{\tau_0}^\tau \|\partial^IL^J\gamma^\mu\mathbf D_\mu[\phi\psi]\|_{L^2(\Sigma_{\tau'})}\,d\tau' \le \frac12 C\varepsilon \tau^{\frac12+(|I'|+|J|)\delta}, \quad N-3\le |I|+|J|\le N, \\
    \int_{\tau_0}^\tau \|\partial^IL^J\gamma^\mu\mathbf D_\mu[\phi\psi]\|_{L^2(\Sigma_{\tau'})}\,d\tau' \le \frac12 C\varepsilon \tau^{(|I'|+|J|)\delta}, \quad  |I|+|J|\le N-4.
    \end{aligned}
\end{align}
In the estimates of the lower-order terms \eqref{bootest-1}, one has to carefully deal with the case when the Lorentz boost $L$ does not act on the field $\phi,\psi$ and yields an additional $t$-factor, which becomes problematic in our energy method. However, this can be remedied via the Hardy-type inequality on the hyperboloid (see Lemma \ref{hardy-ineq}) and the specific structure of the energy: $\|\partial_j\psi+\frac{x^j}{t}\partial_t\psi\|_{L^2(\Sigma_\tau)}\le E_c[\psi](\tau)^\frac12$ and $\|\partial_t\psi\|_{L^2(\Sigma_\tau)}\le \frac{t}{\tau}E_c[\psi](\tau)^\frac12$ for any $c>0$. This is one of the main concerns of Section \ref{sec:commutators}.

In the proof of \eqref{bootest-2} and \eqref{bootest-3}, we further reduce the estimates of the nonlinear terms $\|\partial^IL^J\gamma^\mu\mathbf D_\mu F\|_{L^2(\Sigma_\tau)}$ involving the Dirac operator to simply $\|\partial_\mu\partial^IL^J F\|_{L^2(\Sigma_\tau)}$ with $\mu=0,1,2,3$. In other words, one can consider the spin-nonlinear estimates to be merely scalar-nonlinear estimates. This is also one of the main concern of Section \ref{sec:commutators} and we refer to Proposition \ref{prop-dirac-reduce}.
\subsection{Non-stationary metric} It is a natural question of whether the stationarity assumption can be relaxed and one can still obtain the global existence on a non-stationary spacetime, which is sufficiently close to the Minkowski space. This is the case when the metric satisfies a mild decay assumption for the time derivatives. Instead of the metric $g$ with \eqref{metric-decay}, we consider a Lorentzian metric $g$ such that
\begin{enumerate}
    \item $t=const.$ is space-like.
    \item $g$ is decomposed into $g=m+g_{sr}+g_{lr}$,
    where $g_{lr}$ is a stationary long range spherically symmetric component and $g_{sr}$ is a short range, not necessarily stationary component, such that
    \begin{align}\label{nonstationary}
        \begin{aligned}
             |\partial_x^\alpha g_{lr}|&\le c^{lr}_\alpha \langle r\rangle^{-1-|\alpha|}, \quad |\alpha|\le N+2, \\
         |\partial_t^\beta\partial_x^\alpha g_{sr}|&\le c^{sr}_{\alpha,\beta} \langle r+t\rangle^{-1-\eta}\langle r\rangle^{-(1-\eta+|\alpha|+\beta)}, \quad |\alpha|+\beta\le N+2,
        \end{aligned}
    \end{align}
    for some $0<\eta<1$.
\end{enumerate}
Under the assumption \eqref{nonstationary}, the global existence and sharp pointwise decay estimates for the cubic Dirac and the (DKG) systems can be established in essentially the same manner as the stationary metric \eqref{metric-decay}. In fact, all key analytic ingredients, such as the energy estimates and commutator bounds remain valid. The necessary modification of the proof is straightforward, and the assumption \eqref{nonstationary} does not affect the essential structure of the argument.

\subsection{Previous and Related works}
On the Minkowski spacetime, nonlinear problems for the Dirac equation have seen tremendous progress. Cubic Dirac equation, given by
\begin{align}\label{cubic-flat}
    \begin{aligned}
        -i\gamma^\mu\partial_\mu\psi+M\psi &= (\psi^\dagger\gamma^0\Gamma\psi)\Gamma\psi,
    \end{aligned}
\end{align}
is the representative toy model, describing the self-interaction of a particle, introduced in \cite{soler,thirring}.
Escobedo - Vega \cite{escobedovega} studied local well-posedness for semilinear Dirac equation of the form \eqref{cubic-flat}. Machihara - Nakanishi - Ozawa \cite{machihara} established global well-posedness and scattering for the massive Dirac, with initial data $H^s(\mathbb R^3)$, $s>1$. When $\Gamma=I_4$, the cubic nonlinearity possesses the null structure. This null form was exploited by Bejenaru - Herr \cite{behe,behe1} to obtain the scattering result for the critical Sobolev data $H^1(\mathbb R^3)$ and $H^\frac12(\mathbb R^2)$ with massive case. Then the massless case is complemented by Bournaveas - Candy \cite{boucan}.

The Dirac field can be coupled with various field. For example, the coupling of the Dirac field with a scalar field results in the Dirac-Klein-Gordon (DKG) system:
\begin{align}
    \begin{aligned}
        -i\gamma^\mu\partial_\mu\psi +M\psi &= \phi\psi, \\
        (\Box-m^2)\phi &= \psi^\dagger\gamma^0\psi.
    \end{aligned}
\end{align}
Ancona - Foschi - Selberg \cite{anfoselb} exploited the null structure in the system and showed almost optimal local well-posedness. This result was improved to global well-posedness with the same regularity by Bejenaru - Herr \cite{beherr}. Then Wang \cite{wang} apply additional regularity in the angular variables to improve those result and to obtain the critical regularity well-posedness. Then Candy - Herr \cite{canhe2} established the scattering result for the system at the critical Sobolev regularity with only a small amount of angular regularity, and then they obtained conditional large data scattering \cite{canhe1,canhe}.

When the Dirac field is coupled with an abelian gauge field $A_\mu$, one has the Maxwell-Dirac (MD) system:
\begin{align}
    \begin{aligned}
        -i\gamma^\mu\partial_\mu\psi+m\psi = A_\mu\gamma^0\gamma^\mu\psi, \\
        \partial^\nu F_{\mu\nu} = -\psi^\dagger\gamma^0\gamma_\mu\psi,
    \end{aligned}
\end{align}
where $F_{\mu\nu}=\partial_\mu A_\nu-\partial_\nu A_\mu$ is the curvature $2$-form associated to the gauge field $A_\mu$. This system is more subtle, since the equation is dependent on the gauge choice, such as the Coulomb gauge $\partial^jA_j=0$ and the Lorenz gauge $\partial^\mu A_\mu=0$. The Coulomb gauge results in a mix of hyperbolic and elliptic equations. Gavrus and Oh \cite{gavrusoh} proved global well-posedness and scattering for the $(1+4)$-dimensional (MD) equation under the Coulomb gauge. Concerning the Lorenz gauge, which leads one to the Lorenz-invariant hyperbolic equations, Ancona - Foschi - Selberg \cite{anfoselb1} proved almost optimal local well-posedness for the $(1+3)$-dimensional (MD) system. Global well-posedness for the $(1+2)$-dimensional system was also obtained in \cite{anselb}. Then modified scattering for the $(1+4)$-dimensional (MD) was proved by Lee \cite{klee}. More recently, Herr - Ifrim - Spitz \cite{herrifrim}, and Cho - Lee \cite{choklee} proved modified scattering for the $(1+3)$-dimensional MD system.

One can also formulate the cubic Dirac equation with the Hartree-type nonlinearity by decoupling the Dirac-Klein-Gordon \cite{chagla} and the Maxwell-Dirac systems \cite{chagla1} as follows:
\begin{align}
    -i\gamma^\mu\partial_\mu\psi+m\psi = V_b*(\psi^\dagger\gamma^0\psi)\psi,
\end{align}
where $V_b(x)=\frac{e^{-b|x|}}{|x|}$, with $b\ge0$. For $b>0$ small scattering result was obtained by Tesfahun \cite{tes,tes1} and Yang \cite{cyang}. Then it was improved in \cite{choleeoz,chohonglee,chohongoz}. For $b=0$ modified scattering result was proven in \cite{cloos} and \cite{chokwonleeyang}. Global large data solutions for the 2D system was obtained in \cite{geosha}. We also refer to related problems concerning the honeycomb structure \cite{klee1}, and semi-relativistic equations (or Boson star equations) \cite{lenz,herrlenz,herrtes,pusa,kwonleeyang}.

We also refer the readers to the work by Huh - Oh \cite{huhoh}, Okamoto \cite{okamoto}, Bournaveas - Candy - Machihara \cite{boucanma}, Pecher \cite{pecher}, proving local well-posedness theory for the Chern-Simons-Dirac system.

Concerning the Dirac equation on a non-trivial geometry, we refer the readers to \cite{parker} as an expository literature. Recently, spectral properties of the Dirac operator were studied by B\"ar \cite{bar}. Finster - Kamran - Smoller - Yau \cite{finster1} investigated the long-time dynamics of of the Dirac equation on the Kerr-Newmann blackhole background, and they obtained sharp decay rates for the massive Dirac on this setting \cite{finster}. The decay rates were also obtained for the massless Dirac on the Schwarzschild background by Smoller - Xie \cite{smollerxie}. This was later improved by Ma - Zhang \cite{mazhang}, which shows the sharp Price's law for the massless Dirac on the Schwarzschild black hole background. Regarding the scattering problems, Nicolas \cite{nicolas} proved scattering for the linear Dirac fields on a spherically symmetric black hole background. Then scattering of massive Dirac on the Schwarzschild background was proved by Jin \cite{jin}. On the Kerr background, scattering was obtained by H\"afner and Nicolas for the massless Dirac \cite{hanicolas}, and by Batic for the massive case \cite{batic}. More recently, conformal scattering results were obtained for the Dirac inside a Reissner-Nordstr\"om-type black hole background by H\"afner, Nicolas, Mokdad \cite{hanimo}, and on the Kerr by Pham \cite{pham}.

 In the context of the dispersive PDEs, the Strichartz estimates \cite{keeltao,strichartz} play a crucial role. We also refer to \cite{ginibrevelo} for an expository literature. Cacciafesta and Suzzoni \cite{cacciasu1} established dispersive inequality on an asymptotically flat spacetime. Then local-in-time Strichartz estimates for the Dirac equation on spherically symmetric space was obtained \cite{cacciasu}, and this was later improved to global-in-time estimates \cite{artzcaccia}. The Strichartz estimates for the Dirac equation were proved on an asymptotically flat spacetime with the non-trapping condition \cite{caccia1}, and on compact manifolds without boundary \cite{caccia2}. More recently, Herr and the present author \cite{herrh} investigated the endpoint Strichartz estimates for the half-Klein-Gordon equation on weakly asymptotically flat spacetime via an outgoing parametrix construction introduced by \cite{tataru4,metataru} and established global well-posedness and scattering for the cubic Dirac equation, covering the whole subcritical regime, $H^s(\mathbb R^3)$, $s>1$.

Concerning nonlinear problems, global existence for the Dirac fields on the de Sitter spacetime with the nonlinearity satisfying the Lipschitz condition was proven by Yagdjian \cite{yag}, and similar results in the FLRW spacetime were known by \cite{galyag}. They also obtained a fundamental solution for the Dirac equation \cite{galyag1}.  Global existence for the massless Maxwell-Dirac system on an asymptotically flat spacetime was established by Ginoux - M\"uller \cite{gimu}. More recently, global stability of the Einstein-Dirac system in the Minkowski solution has been proved by Chen for the massless case \cite{chen}.

Meanwhile, the study of the wave and Klein-Gordon equations has seen remarkable progress over the past decades, both in the flat and curved settings, including global existence, scattering, and sharp decay estimates. 
Rather than providing an exhaustive list, we refer the readers to a few representative works, including \cite{klai,christodoulou,christoklai,hoermander,metato,daferrod1,daferrod,dahorod,dahorod1}, and references therein.
  

\subsection{Toy model: nonlinear Klein-Gordon equations} We would like to emphasise that the key argument of Theorem \ref{thm-global-cubic} and Theorem \ref{thm-global-dkg} is mainly derived from the well-known argument concerning scalar field equations, and it can be directly applied to several nonlinear problems for the scalar field equations. Here we present cubic and quadratic nonlinear Klein-Gordon equations on an asymptotically flat spacetime, close to the Minkowski spacetime:
\begin{align}\label{toy-cubic-kg}
\begin{aligned}
	(\Box_g-1)\phi = |\phi|^2\phi, \\
	(\phi,\partial_t\phi)|_{t=t_0} := (\phi_0,\phi_1),
\end{aligned}	
\end{align}
and
\begin{align}\label{toy-quadra-kg}
\begin{aligned}
	(\Box_g-1)\phi = \phi\nabla\phi, \\
	(\phi,\partial_t\phi)|_{t=t_0} := (\phi_0,\phi_1),
\end{aligned}	
\end{align}
where $\nabla=\nabla_{t,x}$ is the space-time derivative.
By repeating the argument as the proof of Theorem \ref{thm-global-cubic} and Theorem \ref{thm-global-dkg}, one is able to establish global existence and pointwise decay of the solutions to Cauchy problems \eqref{toy-cubic-kg} and \eqref{toy-quadra-kg}, respectively. Moreover, the quadratic term with at most first-order derivative does not require any specific null condition. As a toy model problem, we present the proof in the Appendix.
\subsection*{Organisation} The rest of this paper is organised as follows. We end this section with the notations, which will be used throughout this paper. Section \ref{sec:dirac-operator} is devoted to the introduction of the Dirac operator. The key of this section is Proposition \ref{prop-squaring-dirac}, which transfers the problem of the Dirac equations to partially the problem of the Klein-Gordon. In Section \ref{sec:foliation} we briefly review the hyperboloidal foliation method, which has been systematically introduced in the works by LeFloch - Ma. Section \ref{sec:commutators} concerns the commutators with the Laplace-Beltrami operator $\Box_g$ and the Dirac operator $\gamma^\mu\mathbf D_\mu$. Finally we prove Theorem \ref{thm-global-cubic} and Theorem \ref{thm-global-dkg} in Section \ref{sec:main-proof} via the bootstrap argument. Appendix is devoted to the detailed discussion of the derivation of the energy inequality on the hyperboloid and the proof of global existence of cubic and quadratic nonlinear Klein-Gordon equations, as a toy model.
\subsection*{Notations}
We first introduce the basic notations on the Cartesian coordinates. We refer the readers to \cite{alinhac}.
We write $(t,x):= (x^0,x^1,x^2,x^3)$ and $r:=|x|=\sqrt{(x^1)^2+(x^2)^2+(x^3)^2}$. We also use the bracket notation $\langle x\rangle=\sqrt{1+|x|^2}$. Furthermore, we denote the partial derivative $\frac{\partial}{\partial x^\mu}$, $\mu=0,1,2,3$ simply by $\partial_\mu$.
Since we are only concerned with an asymptotically flat spacetime, which is close to the Minkowski spacetime throughout this paper, it is still reasonable to consider the coordinate $x^0=t$ as the time, and the coordinates $(x^1,x^2,x^3)$ as the spatial coordinates.

Vector fields $V,W,\cdots$ defined on $\mathbb R^{1+3}$ are denoted by $V=V^\mu\partial_\mu$ and $W=W^\nu\partial_\nu$, while $1$-forms $\omega,\eta,\cdots$ are written by $\omega=\omega_\mu dx^\mu$ and $\eta=\eta_\nu dx^\nu$ in terms of the Cartesian coordinates.

A metric $g$ is the smooth assignment to each point $x$ of a symmetric non-degenerate bilinear form on the tangent space $T_x\mathbb R^{1+3}$. In the coordinates $x^\mu$, the components of the metric are $g_{\mu\nu}=g(\partial_\mu,\partial_\nu)$, which are supposed to be smooth. Hence the metric $g$ can be identified with the symmetric $4\times4$ invertible matrix $[g_{\mu\nu}]$. The elements of the inverse metric are denoted by $g^{\mu\nu}$. Throughout this paper, the signature of the quadratic form $g$ is assumed to be $(-1,+1,+1,+1)$, which means that the metric $g$ is Lorentzian.

We adopt the Einstein summation convention, i.e., repeated indices implicitly denote summation. From now on the Greek indices $\mu,\nu,\lambda$ range over $0,1,2,3$, whereas the Latin indices $i,j,k$ range over $1,2,3$. For example, by the notation $\gamma^\mu\mathbf D_\mu$ we mean the summation $\sum_{\mu=0}^3\gamma^\mu\mathbf D_\mu$ and $\gamma^j\mathbf D_j = \sum_{j=1}^3 \gamma^k \mathbf D_k$.

In what follows, we will raise and lower the index by using the metric tensor $g$. For example, we write $x_\mu = g_{\mu\nu}x^\nu$ and $\partial^\mu\phi = g^{\mu\nu}\partial_\nu\phi$. Using the repeated indices, the metric $g$ is also written as
\begin{align*}
    g = g_{\mu\nu}dx^\mu dx^\nu , \quad g(V,W) = g_{\mu\nu}V^\mu W^\nu = V_\mu W^\mu = V^\nu W_\nu.
\end{align*}
We denote the Lorentz boost by $L_i = t\partial_i+x^i\partial_t$, $i=1,2,3$. Now given any multi-index $I=(\mu_n,\mu_{n-1},\cdots,\mu_1)$ we denote by $\partial^I = \partial_{\mu_n}\partial_{\mu_{n-1}}\cdots\partial_{\mu_1}$ the product of $n=|I|$ partial derivatives with $0\le \mu_i \le 3$. Similarly, for any multi-index $J=(j_m,j_{m-1},\cdots,j_1)$ we denote by $L^J=L_{j_m}\cdots L_{j_1}$, with $1\le j_{k}\le 3$, the product of $m=|J|$ Lorentz boosts.

For two positive numbers $A$ and $B$, we write $A\lesssim B$ if $A\le CB$, for some absolute constant $C$ (which only depends on irrelevant parameters). If $C$ can be chosen sufficiently small, we write $A\ll B$.
Also, we write $A\approx B$ if both $A\lesssim B$ and $B\lesssim A$.

Given a complex number $z\in \mathbb C$ we denote its real part by $\Re z$.
\section{Dirac operators}\label{sec:dirac-operator}
The homogeneous covariant Dirac equation is given by
\begin{align}\label{eq-dirac-cov}
	(-i\gamma^\mu(x)\mathbf D_\mu+M)\psi = 0,
\end{align}
where $\psi:\mathbb R^{1+3}\to\mathbb C^4$ is the unknown spinor field and $M>0$ is a mass of the particle.
On a generally curved spacetime, the gamma matrices $\gamma^\mu$, $\mu=0,\cdots,3$ are no longer constant matrices. Instead, the matrices $\gamma^\mu=\gamma^\mu(x)$ also become space(-time) dependent complex matrices in a stationary (a non-stationary) metric. The algebraic property is then generalised to 
\begin{align}\label{gamma-alge-curv}
\gamma^\mu\gamma^\nu+\gamma^\nu\gamma^\mu = -2g^{\mu\nu}(x) I_{4}.	
\end{align}
We define the covariant derivative acting on a spinor field $\psi:\mathbb R^{1+3}\rightarrow\mathbb C^4$ to be
\begin{align}
\mathbf D_\mu \psi = (\partial_\mu-\Gamma_\mu)\psi,	
\end{align}
where the spinorial affine connections $\Gamma_\mu=\Gamma_\mu(x)$ are matrices defined by the vanishing of the covariant derivative of the gamma matrices:
\begin{align}\label{affine-spin}
	\mathbf D_\mu \gamma_\nu = \partial_\mu\gamma_\nu-\Gamma^\lambda_{\mu\nu}\gamma_\lambda-\Gamma_\mu\gamma_\nu+\gamma_\nu\Gamma_\mu = 0.
\end{align}
One can represent the $\gamma^\mu(x)$ matrices in terms of the gamma matrices $\tilde{\gamma}^\mu$ on the flat spacetime by introducing a \textit{vierbein} $b^\alpha_\mu(x)$ of vector fields, defined by
\begin{align}\label{metric-vierbein}
	g_{\mu\nu}(x) = m_{\alpha\beta}b^\alpha_\mu(x)b^\beta_\nu(x).
\end{align}
Then the matrices $\gamma^\mu(t,x)$ can be written in the form
\begin{align}
\gamma^\mu(x) = b^\mu_\alpha(x)\tilde{\gamma}^\alpha.	
\end{align}
Here we used the notation $\gamma^\mu(x)$ and $\tilde\gamma^\mu$ to distinguish the gamma matrices dependent on space-time from the constant matrices.
Now a set of $\Gamma_\mu(x)$ satisfying \eqref{affine-spin} are given by
\begin{align}\label{formula-Gamma}
\Gamma_\mu(x) = -\frac14 \tilde{\gamma}_\alpha\tilde{\gamma}_\beta b^{\alpha\lambda}(x)\mathbf D_\mu b^\beta_\lambda(x),	
\end{align}
where
\begin{align*}
	\mathbf D_\mu b^\beta_\lambda = \partial_\mu b^\beta_\lambda - \Gamma^\sigma_{\mu\lambda}b^\beta_\sigma.
\end{align*}
We refer the readers to \cite[Section 3.9]{parker} for the details. However, concerning the spinorial affine connections $\Gamma_\mu$, we only need the inequalities $|\Gamma_\mu(x)|\lesssim |\partial_{x}g_{}|$. See also Proposition 2.1 of \cite{caccia1}. Note that $\Gamma_\mu$ behaves as $1$-forms. This will be exploited in the proof of Proposition \ref{prop-squaring-dirac}. 

Since the Dirac equation was first developed via the factorisation of the Klein-Gordon equation on the Minkowski spacetime, squaring the Dirac operator leads us directly to the Klein-Gordon equation for each of the spinor components. Obviously, this is not the case for the Dirac equation on a curved spacetime. However, one can still obtain a generalisation of the Klein-Gordon equation up to spinorial connections. More precisely, we have the following: (See also \cite{alcu} and \cite{parker}.)
\begin{prop}\label{prop-squaring-dirac}
Given an inhomogeneous Dirac equation $-i\gamma^\mu\mathbf D_\mu\psi+M\psi = F$, we have
\begin{align}
(\Box_g-M^2)\psi = -2\Gamma^\mu\partial_\mu\psi +V\psi -i\gamma^\mu\mathbf D_\mu F-MF,
\end{align}
where $\Box_g$ is the Laplace-Beltrami operator associated to the metric $g$ and the potential $V$ satisfies the inequality $|V(x)|\lesssim |\partial^2 g|+|\partial g|^2$.
\end{prop}
\begin{proof}
To see this, we write
\begin{align*}
\gamma^\mu\mathbf D_\mu(\gamma^\nu\mathbf D_\nu \psi ) & = \gamma^\mu\gamma^\nu\mathbf D_\mu\mathbf D_\nu\psi  \\
& = \frac12( \gamma^\mu\gamma^\nu+\gamma^\nu\gamma^\mu+\gamma^\mu\gamma^\nu-\gamma^\nu\gamma^\mu)\mathbf D_\mu\mathbf D_\nu\psi,
\end{align*}
where we used the fact $\mathbf D_\mu\gamma^\nu=0$.
By the algebraic relation for the gamma matrices \eqref{gamma-alge-curv}, we have
\begin{align*}
	\frac12(\gamma^\mu\gamma^\nu+\gamma^\nu\gamma^\mu)\mathbf D_\mu\mathbf D_\nu \psi & = -g^{\mu\nu}\mathbf D_\nu\mathbf D_\nu\psi .
\end{align*}
On the other hand, we relabel the index to get
\begin{align*}
	\frac12(\gamma^\mu\gamma^\nu-\gamma^\nu\gamma^\mu)\mathbf D_\mu\mathbf D_\nu\psi & = \frac14\left( (\gamma^\mu\gamma^\nu-\gamma^\nu\gamma^\mu)\mathbf D_\mu\mathbf D_\nu+(\gamma^\nu\gamma^\mu-\gamma^\mu\gamma^\nu)\mathbf D_\nu\mathbf D_\nu \right)\psi \\
	& = \frac14 \left( \gamma^\mu\gamma^\nu(\mathbf D_\mu\mathbf D_\nu-\mathbf D_\nu\mathbf D_\mu )-\gamma^\nu\gamma^\mu(\mathbf D_\mu\mathbf D_\nu-\mathbf D_\nu\mathbf D_\mu) \right) \psi \\
	& = \frac14[\gamma^\mu,\gamma^\nu][\mathbf D_\mu,\mathbf D_\nu]\psi.
\end{align*}
Now we use the following identities:
\begin{align}
\begin{aligned}
   &[	\mathbf D_\mu,\mathbf D_\nu ] \psi = \frac14 \gamma^\lambda\gamma^\sigma R_{\mu\nu\lambda\sigma}\psi, \\
&\gamma^\mu\gamma^\nu\gamma^\lambda\gamma^\sigma R_{\mu\nu\lambda\sigma} = -2\mathbf R,
\end{aligned}	
\end{align}
where $\mathbf R$ is the scalar curvature associated to the metric $g$. We refer the readers to Section 5.6.1 of Parker. From these identites we conclude that
\begin{align*}
	\frac14[\gamma^\mu,\gamma^\nu][\mathbf D_\mu,\mathbf D_\nu]\psi & = -\frac14\mathbf R\psi.
\end{align*}
Finally we obtain
\begin{align*}
	\gamma^\mu\mathbf D_\mu(\gamma^\nu\mathbf D_\nu \psi )  = -g^{\mu\nu}\mathbf D_\mu\mathbf D_\nu\psi-\frac14\mathbf R\psi.
\end{align*}
Now we compute the term $\mathbf D_\nu\mathbf D_\nu\psi$. We first recall that $\mathbf D_\mu\psi$ is also a $1$-form. Then
\begin{align*}
	\mathbf D_\mu\mathbf D_\nu\psi & = \partial_\mu(\mathbf D_\nu\psi)-\Gamma_\mu(\mathbf D_\nu\psi)-\Gamma^\lambda_{\mu\nu}\mathbf D_\lambda\psi \\
	& = \partial_\mu(\partial_\nu\psi-\Gamma_\nu\psi)-\Gamma_\mu(\partial_\mu\psi-\Gamma_\nu\psi)-\Gamma^\lambda_{\mu\nu}(\partial_\lambda\psi-\Gamma_\lambda\psi) \\
	& = (\partial_\mu\partial_\nu\psi-\Gamma^\lambda_{\mu\nu}\partial_\lambda\psi)-\Gamma_\mu\partial_\nu\psi-\Gamma_\nu\partial_\mu\psi -(\partial_\mu\Gamma_\nu)\psi+\Gamma^\lambda_{\mu\nu}\Gamma_\lambda\psi +\Gamma_\mu\Gamma_\nu\psi \\
	& = \nabla_\mu\nabla_\nu\psi-(\Gamma_\mu\partial_\nu\psi+\Gamma_\nu\partial_\mu\psi)-(\nabla_\mu\Gamma_\nu-\Gamma_\mu\Gamma_\nu)\psi,
\end{align*}
where $\nabla_\mu$ is the usual covariant derivative that acts on $\phi$ as a scalar, and on $\Gamma_\mu$ as a $1$-form. See also \cite{alcu}. Therefore we see that
\begin{align*}
	-g^{\mu\nu}\mathbf D_\mu\mathbf D_\nu\psi & = -\nabla^\mu\nabla_\mu\psi-2\Gamma^\mu\partial_\mu\psi -(\nabla^\mu\Gamma_\mu-\Gamma^\mu\Gamma_\mu)\psi,
\end{align*}
where $\nabla^\mu\nabla_\mu$ is then the Laplace-Beltrami operator $\Box_g$ associated to the metric $g$. In consequence, an application of the operator $-i\gamma^\mu\mathbf D_\mu$ on the equation $-i\gamma^\nu\mathbf D_\nu\psi+M\psi = F$ yields the desired identity.
\end{proof}
\section{Hyperboloidal foliation method}\label{sec:foliation}
This section is devoted to the discussion on the hyperboloidal foliation method.
From now on we only consider $t\ge2$ inside the forward light cone $\{(t,x):|x|\le t-1\}$ and we assume that $\phi$ is supported inside the forward light cone, i.e., $\phi$ vanishes unless $\{|x|< t-1\}$. In other words, we focus on the solutions which are spatially compactly supported.

In what follows, we briefly review the foliation, energy estimates, and the Klainerman-Sobolev inequality, and Hardy-type estimates on the hyperboloid, introduced in \cite{flochma}. For interested readers, we present in the Appendix the background and the derivation of energy flux and energy estimates using a standard argument in details.

We introduce the foliation $\Sigma_\tau=\{(t,x): \tau = \sqrt{t^2-|x|^2}\}$. 
We define the energy on the hyperboloid:
\begin{align}
    E^{}[\phi](\tau) = \int_{\Sigma_\tau}\left( |\partial_t\phi|^2+|\partial_x\phi|^2+|\phi|^2+2\frac{x^j}{t}\Re(\partial_t\phi\overline{\partial_j\phi}) \right)\,dx,
\end{align}
or equivalently,
\begin{align}
    E^{}[\phi](\tau) = \int_{\Sigma_\tau} \sum_{j=1}^3 \left| \partial_j\phi+\frac{x^j}{t}\partial_t\phi \right|^2+\left| \frac{\tau}{t}\partial_t\phi \right|^2+|\phi|^2 \,dx.
\end{align}
We denote the integral of a function $\phi$ along the hyperboloid $\Sigma_\tau$ by
\begin{align*}
    \|\phi\|_{L^p(\Sigma_\tau)}^p := \int_{\Sigma_\tau}|\phi|^p\,dx = \int_{\mathbb R^3}|\phi(\sqrt{\tau^2+|x|^2},x)|^p\,dx.
\end{align*}
We also set the domain $D_{[\tau_0,\tau_1]}$ to be
\begin{align*}
    D_{[\tau_0,\tau_1]} := \bigcup_{\tau_0\le\tau\le\tau_1}\Sigma_{\tau} \,\cap \{(t,x): |x|<t-1\}. 
\end{align*}
\begin{prop}\label{energy-ineq}
For the solution $\phi$ to the inhomogeneous Klein-Gordon equation $(\Box_g-1)\phi =F$, which are spatially supported inside the forward light cone $\{(t,x): |x|<t-1\}$, we have the energy estimates:
    \begin{align}
         \sup_{\tau_0\le \tau'\le \tau}E^{}[\phi](\tau')^\frac12 \lesssim E^{}[\phi](\tau_0)^\frac12+\int_{\tau_0}^\tau \|F\|_{L^2(\Sigma_{\tau'})}\,d\tau'.
    \end{align}
\end{prop}
The proof follows from a standard argument of using the divergence theorem with the contraction of the energy-momentum tensor and the time-like vector field $\partial_t$. We refer the readers to the Appendix for the details.

Now we have the hyperboloid version of the Klainerman-Sobolev inequality: See also Lemma 7.6.1 of \cite{hoermander} and Proposition 5.1.1 of \cite{flochma1}.
\begin{lem}[Proposition 2.2 of \cite{flochma}]\label{KS-ineq}
Let $\phi$ be a sufficiently smooth function, which is spatially supported inside the forward light cone $\{(t,x): |x|< t-1\}$. Then we have
\begin{align*}
	\sup_{(t,x)\in\Sigma_\tau} t^\frac32 |\phi(t,x)| \lesssim \sum_{|J|\le2}\|L^J\phi\|_{L^2(\Sigma_\tau)},
\end{align*}
where $L\in\{L_i\}_{i=1,2,3}$ is the Lorentz boost.
\end{lem}
\begin{lem}[Lemma 2.4 of \cite{flochma}]\label{hardy-ineq}
    For any sufficiently smooth function $\phi$ which is defined in the forward region $D_{[2,\tau]}$ and is spatially supported inside the forward light cone $\{(t,x):|x|<t-1\}$, we have for all $\tau\ge2$, 
    \begin{align}
        \|r^{-1}\phi\|_{L^2(\Sigma_\tau)} \lesssim \sum_{i=1}^3\|t^{-1}L_i\phi\|_{L^2(\Sigma_\tau)}.
    \end{align}
\end{lem}
The Hardy-type inequality will be effectively used to gain $t^{-1}$-factor from the error terms which appear in the commutator of the Lorentz boost and $\Box_g$ and the Dirac operator $\gamma^\mu\mathbf D_\mu$, which is a core of the following section.

\section{Commutators}\label{sec:commutators}
In this section, we present several commutators with the Laplace-Beltrami operator $\Box_g$ and the Dirac operator $\gamma^\mu\mathbf D_\mu$. The main issue is the commutator with the Lorentz boost $L_i=t\partial_i+x^i\partial_i$, $i=1,2,3$, in that it gives an additional $t$-factor in the error terms, which do not appear in the Minkowski spacetime. This becomes problematic in the energy estimates, and hence the aim of this section is to discuss how one can eliminate this problematic $t$-growth from commutators $[\Box_g,L_i]$ and $[\gamma^\mu\mathbf D_\mu,L_i]$. Moreover, we also reduce the nonlinear problem involving $\gamma^\mu\mathbf D_\mu$ to merely a nonlinear problem of scalar fields.
\subsection{Commutators with the Laplace-Beltrami operator}
The commutator of the Laplace-Beltrami operator $\Box_g$ with the translation operator $\partial_\mu$ is obvious, and we omit the details. Now we focus on the commutator of $\Box_g$ with the Lorentz boost. We first recall that
\begin{align*}
	\Box_g = \Box+g^\omega_{lr}(r)\Delta_\omega+\partial_\alpha g^{\alpha\beta}_{sr}\partial_\beta 
\end{align*}
where $g^\omega_{lr}(r)\approx \langle r\rangle^{-3}$ and $g^{00}_{sr}=0$. Then the commutator is given by
\begin{align*}
	[\Box_g,L_{i}] = [\partial_\mu g^{\mu\nu}_{sr}\partial_\nu, t\partial_i+x^i\partial_t].
		\end{align*}
Then
\begin{align*}
	[\partial_jg^{jk}_{sr}\partial_k,L_{i}] & = [(\partial_jg^{jk}_{sr})\partial_k,t\partial_i+x^i\partial_t] + [g^{jk}_{sr}\partial_j\partial_k,t\partial_i+x^i\partial_t],
\end{align*}
and
\begin{align*}
[g^{0j}_{sr}\partial_j\partial_t,L_{i}] & = g^{0j}_{sr}\partial_j\partial_i-t(\partial_ig^{0j}_{sr})\partial_j\partial_t+g^{0j}_{sr}\partial_t^2, \\
	[(\partial_jg^{jk}_{sr})\partial_k,L_{i}]  & = -t(\partial_i\partial_jg^{jk}_{sr})\partial_k+(\partial_jg^{ji}_{sr})\partial_t,\\
	[g^{jk}_{sr}\partial_j\partial_k,L_{i}] & = -t(\partial_ig^{jk}_{sr})\partial_j\partial_k+g^{ji}_{sr}\partial_t\partial_j+g^{ik}_{sr}\partial_t\partial_k.
\end{align*}
Thus one can encounter an additional $t$ factor from the commutator of $\Box_g$ with the Lorentz boost $L_i$. This situation becomes even worse when one has the commutator of $\Box_g$ with the product of $|J|$-Lorentz boost $L^J$, $|J|\ge2$, and all of $L^J$ act on the metric, which yields the factor $t^{|J|}$.
 However, we can remedy this dangerous $t$-growth. It is helpful to introduce the bootstrap assumption for a moment:

For $|I|+|J|\le N$ with some fixed $N\ge11$, we assume that the energy bound $E[\partial^IL^J\phi](\tau)^\frac12 \le C\varepsilon\tau^{\frac12+(|I'|+|J|)\delta}$ holds on a hyperbolic time interval $[\tau_0,\tau_1]$ for $|I|+|J|\le N$, where $\partial^I=\partial_t^{I'}\partial_x^{I''}$. We will discuss this bootstrap assumption in detail in Section \ref{sec:main-proof}. 
\begin{prop}\label{prop-error-box}
Let $\phi$ be a sufficiently smooth and spatially compactly supported function defined inside the forward light cone $\{(t,x): |x|<t-1\}$. 
Under the aforementioned bootstrap assumptions, the commutator term $[\partial^IL^J,\Box_g]\phi$ satisfies for $|I|+|J|\le N$,
\begin{align}
\|[\partial^IL^J,\Box_g]\phi\|_{L^2(\Sigma_\tau)} \le C\varepsilon t^{-\frac12}\tau^{-\frac12+(N+2)\delta}.	
\end{align}
\end{prop}
To see this we first note that for $|I|+|J|=N$
the commutator $[\partial^IL^J,\Box_g]$ yields the summation
\begin{align*}
	[\partial^IL^J,\Box_g]\phi & = \sum_{\substack{I_1,I_2,J_1,J_2 \\ |I_2|+|J_2|\le N-1 }} (\partial^{I_1}L^{J_1}  g^{\mu\nu}_{sr})\partial_\mu\partial_\nu (\partial^{I_2}L^{J_2}\phi)+ \sum_{\substack{I_1,I_2,J_1,J_2 \\ |I_2|+|J_2|\le N-1 }} (\partial^{I_1}L^{J_1}  \partial_\mu g^{\mu\nu}_{sr})\partial_\nu (\partial^{I_2}L^{J_2}\phi).
\end{align*}
From now on, we consider the case $|I_2|+|J_2|=N-1$. For brevity, we put $t^{-1}L_i=V_i$.
If $|J_1|=1$ and $|I_1|=0$, with $(\mu,\nu)=(0,i)$ then
\begin{align*}
	|L(g^{0i}_{sr})\partial_t\partial_i \partial^{I_2}L^{J-1}\phi| & \le t\langle r\rangle^{-3}|\partial_t\partial_i \partial^IL^{J-1}\phi| \\
	& = t\langle r\rangle^{-3}|\partial_t(V_i-\frac{x^i}{t}\partial_t) \partial^IL^{J-1}\phi |\\
	& \le t\langle r\rangle^{-3}|\partial_tV_i \partial^IL^{J-1}\phi|+t\langle r\rangle^{-3}\frac{|x|}{t}|\partial_t^2\partial^IL^{J-1}\phi|,
    \end{align*}
    where we omit an obvious error term which appears due to the term $\bigg|\partial_t\dfrac{x^i}{t}\bigg|\le \dfrac{|x|}{t^2}$. Using $tV_i=L_i$, we see that
    \begin{align*}
&|L(g^{0i}_{sr})\partial_t\partial_i \partial^{I_2}L^{J-1}\phi|\\
	& \le \langle r\rangle^{-3}|\partial_t \partial^{I-2}\partial_j^2L^{J}\phi|+\langle r\rangle^{-2}|\partial_t^2\partial^{I-2}\partial_j^2 L^{J-1}\phi| \\
	& = \langle r\rangle^{-3}|\partial_t \partial^{I-2}(V_j+\frac{x^j}{t}\partial_t)^2L^{J}\phi|+\langle r\rangle^{-2}|\partial_t^2\partial^{I-2}(V_j+\frac{x^j}{t}\partial_t)^2 L^{J-1}\phi| \\
	& \le \langle r\rangle^{-3}|\partial_t\partial^{I-2}V_j^2L^J\phi|+\langle r\rangle^{-3}\frac{|x|}{t}|\partial_t^2\partial^{I-2}V_jL^J\phi|+\langle r\rangle^{-3}\frac{|x|^2}{t^2}|\partial_t^3\partial^{I-2}L^J\phi |\\
	& \qquad +\langle r\rangle^{-2}|\partial_t^2\partial^{I-2}V_j^2L^{J-1}\phi| +\langle r\rangle^{-2}\frac{|x|}{t}|\partial_t^3\partial^{I-2}V_jL^{J-1}\phi|+\langle r\rangle^{-3}\frac{|x|^2}{t^2}|\partial_t^4\partial^{I-2}L^{J-1}\phi| \\
	& \lesssim t^{-2}\langle r\rangle^{-1}( |\partial^{I-2}L^{J+2}\partial_t\phi| + |\partial^{I-2}L^{J+1}\partial_t^2\phi|+ |\partial^{I-2}L^J\partial_t^3\phi|\\
    &\qquad\qquad\qquad\quad+|\partial^{I-2}L^{J+1}\partial_t^2\phi|+ |\partial^{I-2}L^J\partial_t^3\phi|+|\partial^{I-2}L^{J-1}\partial_t^4\phi|),
\end{align*} 
where we omit several obvious error terms arising from the commutator with $V_i$.
On the other hand, if $J_1=0$ and $I_1=1$, then there is no $t$-growth and we have nothing to do. The case $(\mu,\nu)=(i,j)$ is similar, which we omit the details. Therefore we see that
\begin{align*}
\|L(g_{sr}^{\alpha\beta})\partial_\alpha\partial_\beta \partial^{I_2}L^{J-1}\phi\|_{L^2(\Sigma_\tau)} & \lesssim t^{-2}C\varepsilon \tau^{\frac12+(|I'|+|J|+2)\delta}\frac{t}{\tau},    
\end{align*}
where we used $\|\partial_t\phi\|_{L^2(\Sigma_\tau)}\le \frac{\tau}{t}E[\phi](\tau)^\frac12$ and the bootstrap assumptions.
This completes the case $|I_2|+|J_2|=N-1$. If $|I_2|+|J_2|=N-2$, with $|J_1|=2$, then using the Hardy inequality Lemma \ref{hardy-ineq} on the hyperboloid
we see that
\begin{align*}
\|	L^2(g^{0j}_{sr}\partial_t\partial_j)\partial^IL^{J-2}\phi \|_{L^2(\Sigma_\tau)} & \le \|t^2\langle r\rangle^{-4}\partial_t\partial_j\partial^IL^{J-2}\phi\|_{L^2(\Sigma_\tau)} \\
& \lesssim \|t^2 V_i\langle r\rangle^{-3}\partial_t\partial_j\partial^IL^{J-2}\phi\|_{L^2(\Sigma_\tau)} \\
& \lesssim \|t\langle r\rangle^{-3}\partial_t\partial_j\partial^IL^{J-1}\phi\|_{L^2(\Sigma_\tau)} \\
& \lesssim \|t\langle r\rangle^{-3}\partial_tV_j\partial^IL^{J-1}\phi\|_{L^2(\Sigma_\tau)}+\|t\langle r\rangle^{-3}\frac{|x|}{t}\partial_t^2\partial^I L^{J-1}\phi\|_{L^2(\Sigma_\tau)}.
\end{align*}
Then the remaining step is identical as the case $|I_2|+|J_2|=N-1$. In summary, for $|J_2|\ge2$, in which case one has $t^{|J_2|}$ in the error terms, we apply the Hardy inequality Lemma \ref{hardy-ineq} $|J_2|$-times to gain the factor $t^{-|J_2|}$ and follow the argument for the case $|J_2|=1$. Finally, an inductive argument and the initial bootstrap assumption give
\begin{align}
    \|[\partial^IL^J,\Box_g]\phi\|_{L^2(\Sigma_\tau)} \le C\varepsilon t^{-1}\tau^{-\frac12+(|I'|+|J|+2)\delta}, \quad |I|+|J|\le N,
\end{align}
where we put $\partial^I=\partial_t^{I'}\partial_x^{I''}$. Therefore we conclude that the commutator terms $[\partial^IL^J,\Box_g]\phi$ become acceptable error terms in the energy estimates.
\subsection{Commutators with Dirac operator} Now we consider the commutators with the Dirac operator $\gamma^\mu\mathbf D_\mu$. Even though we make the use of squaring the Dirac operator and transfer the study of the nonlinear Dirac equations to the nonlinear problems of the Klein-Gordon equations as in Proposition \ref{prop-squaring-dirac}, one still has to deal with the Dirac operator, since the Dirac operator $\gamma^\mu\mathbf D_\mu$ then acts on the nonlinearity. Due to the spinor structure, the commutator does not behave in the usual way as scalar fields. However, it turns out that the nonlinear problems associated to the Dirac operator $\gamma^\mu\mathbf D_\mu$ are reduced to the problem of scalar fields. To be precise, we observe the following:
\begin{prop}\label{prop-dirac-reduce}
For $|I|+|J|\le N$, we have
\begin{align}
\|\partial^IL^J\gamma^\mu\mathbf D_\mu\psi\|_{L^2(\Sigma_\tau)} \lesssim \sum_{|I_1|\le |I|}\|\partial_t \partial^{I_1}L^J\psi\|_{L^2(\Sigma_\tau)}+ \sum_{j=1}^3\sum_{|I_1|\le |I|}\|\partial_j \partial^{I_1}L^J\psi\|_{L^2(\Sigma_\tau)}	.
\end{align}
\end{prop}
We shall prove this via several steps. We first consider the commutator of $\gamma^\mu\mathbf D_\mu$ with $\partial_\nu$ and $L_i$.
The commutator with the translation vector fields $\partial_\nu$, $\nu=0,1,2,3$ is somewhat obvious.
\begin{prop}\label{prop-com-dirac-trans}
For the translation vector fields $\partial_\nu$, $\nu=0,1,2,3$, we have
\begin{align}
[\partial_\nu,\gamma^\mu\mathbf D_\mu]\psi = 	 \Gamma_\nu(\gamma^\mu\mathbf D_\mu\psi)-\gamma^\mu(\partial_\nu\Gamma_\mu)\psi.
\end{align}
\end{prop}
This commutator identity follows via a straightforward computation. Indeed, we have
\begin{align*}
	[\partial_\nu,\gamma^\mu\mathbf D_\mu] & = \partial_\nu(\gamma^\mu\mathbf D_\mu)-\gamma^\mu\mathbf D_\mu\partial_\nu \\
	& = (\partial_\nu\gamma^\mu)\mathbf D_\mu+\gamma^\mu\partial_\nu\mathbf D_\mu-\gamma^\mu\mathbf D_\mu\partial_\nu \\
	& = (\partial_\nu\gamma^\mu)\mathbf D_\mu+\gamma^\mu\partial_\nu(\partial_\mu-\Gamma_\mu)-\gamma^\mu(\partial_\mu-\Gamma_\mu)\partial_\nu \\
	& = (\partial_\nu\gamma^\mu)\mathbf D_\mu+\gamma^\mu\partial_\nu\partial_\mu-\gamma^\mu(\partial_\nu\Gamma_\mu)-\gamma^\mu\Gamma_\mu\partial_\nu-\gamma^\mu\partial_\mu\partial_\nu+\gamma^\mu\Gamma_\mu\partial_\nu \\
	& = [(\partial_\nu-\Gamma_\nu)\gamma^\mu]\mathbf D_\mu+\Gamma_\nu\gamma^\mu\mathbf D_\mu-\gamma^\mu(\partial_\nu\Gamma_\mu) \\
	& = \Gamma_\nu(\gamma^\mu\mathbf D_\mu)-\gamma^\mu(\partial_\nu\Gamma_\mu),
\end{align*}
where we used the fact $\mathbf D_\nu\gamma^\mu=0$.
For the commutator with the Lorentz boost $L_i$ we have to carefully consider the spinor structure. Indeed, we have the commutator identity
\begin{prop}\label{prop-com-dirac-lorentz}
Let $L_i=t\partial_i+x^i\partial_t$, $i=1,2,3$, be the Lorentz boost. Then we have
\begin{align}
L_i \gamma^\mu\mathbf D_\mu\psi = 	 \gamma^\mu\mathbf D_\mu L_i\psi+(L_i\gamma^\mu\mathbf D_\mu)\psi + (g_{ij}-\delta_{ij})\gamma^j\partial_t\psi-g_{j0}\gamma^j\partial_i\psi.
\end{align}
\end{prop}
 An obvious computation gives
\begin{align*}
    L_i\gamma^\mu\partial_\mu\psi & = (t\partial_i+x^i\partial_t)\gamma^\mu\partial_\mu\psi  \\
    & = t(\partial_i\gamma^\mu)\partial_\mu\psi+t\gamma^\mu\partial_i\partial_\mu\psi+x^i(\partial_t\gamma^\mu)\partial_\mu\psi+x^i\gamma^\mu\partial_t\partial_\mu\psi,
\end{align*}
and we also observe that 
\begin{align*}
t\gamma^\mu\partial_i\partial_\mu\psi & = t\gamma^\mu\partial_\mu\partial_i\psi \\
    & = t\gamma^0\partial_t\partial_i\psi+t\gamma^j\partial_j\partial_i\psi \\
    & = \gamma^0\partial_t(t\partial_i\psi)-\gamma^0\partial_i\psi+\gamma^j\partial_j(t\partial_i\psi) \\
    & = \gamma^\mu\partial_\mu(t\partial_i\psi)-\gamma^0\partial_i\psi,
\end{align*}
and
\begin{align*}
    x^i\gamma^\mu\partial_t\partial_\mu\psi & = x^i\gamma^\mu\partial_\mu\partial_t\psi \\
    & = x^i\gamma^0\partial_t\partial_t\psi +x^i\gamma^j\partial_j\partial_t\psi \\
    & = \gamma^0\partial_t(x^i\partial_t\psi)+\gamma^j\partial_j(x^i\partial_t\psi)-\gamma^j\delta^i_j\partial_t\psi \\
    & = \gamma^\mu\partial_\mu(x^i\partial_t\psi)-\gamma^i\partial_t\psi.
\end{align*}
Thus we conclude that the commutator of the Dirac operator with the Lorentz boost yields not only the error terms from the derivatives of the gamma matrices but also additional terms $-\gamma^0\partial_i\psi-\gamma^i\partial_t\psi$. To remedy this error term, we introduce the modified Lorentz boost $\widehat{L_i}=L_i+\frac12\gamma_0\gamma_i$. Indeed, we see that
\begin{align*}
    \gamma_0\gamma_i\gamma^\mu & = \gamma_0\gamma^\nu g_{i\nu}\gamma^\mu \\
    & = \gamma_0(-\gamma^\mu\gamma^\nu-2g^{\mu\nu}I_4)g_{i\nu} \\
    & = -\gamma_0\gamma^\mu\gamma^\nu g_{i\nu}-2\gamma_0g^{\mu\nu}g_{i\nu} \\
    & = -\gamma^\lambda g_{\lambda0}\gamma^\mu\gamma^\nu g_{i\nu}-2\gamma_0g^{\mu\nu}g_{i\nu} \\
    & = -(-\gamma^\mu\gamma^\lambda-2g^{\mu\lambda}I_{4})g_{\lambda0}\gamma^\nu g_{i\nu}-2\gamma_0\delta^\mu_i \\
    & =\gamma^\mu\gamma^\lambda\gamma^\nu g_{\lambda0}g_{i\nu}+2g^{\mu\lambda}\gamma^\nu g_{\lambda0}g_{i\nu}-2\gamma_0\delta^\mu_i \\
    & = \gamma^\mu\gamma_0\gamma_i+2\delta^\mu_0\gamma_i-2\gamma_0\delta^\mu_i,
\end{align*}
which implies the identity $\gamma_0\gamma_i\gamma^\mu\partial_\mu = \gamma^\mu\gamma_0\gamma_i\partial_\mu+2\gamma_i\partial_t-2\gamma_0\partial_i $.
Then we have
\begin{align*}
    (L_i+\frac12\gamma_0\gamma_i)\psi &= t\gamma^\mu\partial_i\partial_\mu\psi+x^i\gamma^\mu\partial_t\partial_\mu\psi+t(\partial_i\gamma^\mu)\partial_\mu\psi+x^i(\partial_t\gamma^\mu)\partial_\mu\psi \\
    & \qquad +\frac12\gamma^\mu\gamma_0\gamma_i\partial_\mu\psi +\gamma_i\partial_t\psi-\gamma_0\partial_i\psi.
\end{align*}
We observe that $\gamma_i\partial_t\psi = g_{\mu i}\gamma^\mu \partial_t\psi = \gamma^i\partial_t\psi+(g_{ij}-\delta_{ij})\gamma^j\partial_t\psi+g_{0i}\gamma^0\partial_t\psi$ and $-\gamma_0\partial_i\psi = -\gamma^\nu g_{\nu0}\partial_i\psi = \gamma^0\partial_i\psi-g_{j0}\gamma^j\partial_i$. 
Then we see that
\begin{align*}
(L_i+\frac12\gamma_0\gamma_i)\gamma^\mu\partial_\mu\psi & = \gamma^\mu\partial_\mu(L_i\psi)+\frac12\gamma^\mu\gamma_0\gamma_i\partial_\mu\psi +t(\partial_i\gamma^\mu)\partial_\mu\psi +x^i(\partial_t\gamma^\mu)\partial_\mu\psi \\
& \qquad\qquad + (g_{ij}-\delta_{ij})\gamma^j\partial_t\psi-g_{j0}\gamma^j\partial_i\psi \\
    & = \gamma^\mu\partial_\mu(L_i+\frac12\gamma_0\gamma_i)\psi +t(\partial_i\gamma^\mu)\partial_\mu\psi +x^i(\partial_t\gamma^\mu)\partial_\mu\psi-\frac12\gamma^\mu\partial_\mu(\gamma_0\gamma_i)\psi \\
    & \qquad\qquad + (g_{ij}-\delta_{ij})\gamma^j\partial_t\psi+g_{0i}\gamma^0\partial_t\psi-g_{j0}\gamma^j\partial_i\psi \\
    & = \gamma^\mu\partial_\mu(L_i+\frac12\gamma_0\gamma_i)\psi+(L_i\gamma^\mu)\partial_\mu\psi-\frac12\gamma^\mu\partial_\mu(\gamma_0\gamma_i)\psi \\
    &\qquad\qquad  + (g_{ij}-\delta_{ij})\gamma^j\partial_t\psi+g_{0i}\gamma^0\partial_t\psi-g_{j0}\gamma^j\partial_i\psi.
\end{align*}
Using the above computation, we see that
\begin{align*}
    L_i\gamma^\mu\mathbf D_\mu\psi & = L_i(\gamma^\mu\partial_\mu\psi-\gamma^\mu\Gamma_\mu\psi) \\
    & = L_i\gamma^\mu\partial_\mu\psi-L_i(\gamma^\mu\Gamma_\mu)\psi-\gamma^\mu\Gamma_\mu L_i\psi \\
    & = \widehat{L_i}\gamma^\mu\partial_\mu\psi-\gamma^\mu\Gamma_\mu L_i\psi-\frac12\gamma_0\gamma_i\gamma^\mu\partial_\mu\psi-L_i(\gamma^\mu\Gamma_\mu)\psi \\
    & = \gamma^\mu\partial_\mu \widehat{L_i}\psi-\gamma^\mu\Gamma_\mu L_i\psi-\frac12\gamma_0\gamma_i\gamma^\mu\partial_\mu\psi-L_i(\gamma^\mu\Gamma_\mu)\psi \\
    & \qquad +(L_i\gamma^\mu)\partial_\mu\psi -g^{j0}\gamma_i\partial_j\psi-\frac12\gamma^\mu\partial_\mu(\gamma_0\gamma_i)\psi \\
    & = \gamma^\mu\partial_\mu\widehat{L_i}\psi-\gamma^\mu\Gamma_\mu(L_i+\frac12\gamma_0\gamma_i)\psi+\frac12\gamma^\mu\partial_\mu\Gamma_\mu\psi-\frac12\gamma_0\gamma_i\gamma^\mu\partial_\mu\psi-L_i(\gamma^\mu\Gamma_\mu)\psi \\
    & \qquad\qquad +(L_i\gamma^\mu)\partial_\mu\psi-\frac12\gamma^\mu\partial_\mu(\gamma_0\gamma_i)\psi+ (g_{ij}-\delta_{ij})\gamma^j\partial_t\psi-g_{j0}\gamma^j\partial_i\psi  \\
    & = \gamma^\mu\mathbf D_\mu \widehat{L_i}\psi+\frac12 \gamma^\mu\Gamma_\mu\gamma_0\gamma_i\psi -\frac12 \gamma_0\gamma_i\gamma^\mu\partial_\mu\psi-\frac12\gamma^\mu\partial_\mu(\gamma_0\gamma_i)\psi + (L\gamma^\mu\mathbf D_\mu)\psi \\
    &\qquad\qquad+ (g_{ij}-\delta_{ij})\gamma^j\partial_t\psi+g_{0i}\gamma^0\partial_t\psi-g_{j0}\gamma^j\partial_i\psi \\
    & = \gamma^\mu\mathbf D_\mu L_i\psi+(L_i\gamma^\mu\mathbf D_\mu)\psi + (g_{ij}-\delta_{ij})\gamma^j\partial_t\psi+g_{0i}\gamma^0\partial_t\psi-g_{j0}\gamma^j\partial_i\psi.
\end{align*}
This completes the proof of Proposition \ref{prop-com-dirac-lorentz}.
Note that an additional $t$-factor appears only in the second term $(L_i\gamma^\mu\mathbf D_\mu)\psi$ of the last equality. We first consider $\mu=j=1,2,3$:\footnote{In fact, we only need to consider $(t\partial_i\gamma^j)\partial_j\psi$. Indeed, since $g^{00}$ is constant, the matrix $\gamma^0$ can be chosen to be a constant matrix and hence $L_i\gamma^0$ vanishes. However, we treat here the matrix $\gamma^0$ as if it is not a constant matrix for a moment in order to show that the argument does not rely on a specific condition on the gamma matrix.}
\begin{align*}
	t|(\partial_i\gamma^j)\partial_j\psi| & \le t|\partial_i\gamma^j|\left| \partial_j\psi+\frac{x^j}{t}\partial_t\psi\right|+ t|\partial_i\gamma^j| \left| \frac{x^j}{t}\partial_t\psi\right| \\
	& \le t |\partial_i\gamma^j| t^{-1}|L_j\psi|+|x||\partial_i\gamma^j||\partial_t\psi| \\
	& \lesssim |L\psi|+|\partial_t\psi|.
\end{align*}
On the other hand, for $\mu=0$, using the Hardy inequality Lemma \ref{hardy-ineq} on the hyperboloid, we have
\begin{align*}
	\|t (\partial_i\gamma^0)\partial_t\psi\|_{L^2(\Sigma_\tau)} & \lesssim \|t \langle r\rangle^{-3}\partial_t\psi\|_{L^2(\Sigma_\tau)} \\ 
	& \lesssim \sum_{i=1}^3\|t \langle r\rangle^{-2}V_i \partial_t\psi\|_{L^2(\Sigma_\tau)} \\
	& \lesssim \| \langle r\rangle^{-2}\partial_t L\psi\|_{L^2(\Sigma_\tau)} \lesssim \frac{t}{\tau} E[L\psi](\tau)^\frac12,
\end{align*}
which also becomes an acceptable error in the nonlinear problems.
Now we extend this argument in an inductive way and prove Proposition \ref{prop-dirac-reduce}. It is useful to write the above identity as follows:
\begin{align*}
	L \gamma^\mu\mathbf D_\mu\psi = \gamma^\mu\mathbf D_\mu L\psi + (L\gamma^\mu\mathbf D_\mu)\psi + (g-m)\partial\psi,
\end{align*} 
where $g$ and $m$ are the abbreviation of the metric $g$ and the Minkoswski metric, respectively. By an inductive argument, it is easily seen that for $|J|\ge2$,
\begin{align}
\begin{aligned}
L^J\gamma^\mu\mathbf D_\mu\psi 	& = \gamma^\mu\mathbf D_\mu L^J\psi+\sum_{\substack{ J_1, J_2  \\ |J_2|\le |J|-1 }}C_{J_1,J_2} (L^{J_1}\gamma^\mu\mathbf D_\mu)(L^{J_2}\psi) \\
&\qquad\qquad\qquad+\sum_{|J_1|+|J_2|=|J|}C_{J_1,J_2}'(L^{J_1}(g-m))(L^{J_2}\partial\psi).
\end{aligned}
\end{align}
By Proposition \ref{prop-com-dirac-trans}, the translation vector fields act on the Dirac operator $\gamma^\mu\mathbf D_\mu$ in the usual way as a scalar field. Hence applying the vector fields $\partial^I$, we easily see that
\begin{align}\label{eq-dirac-vf}
\begin{aligned}
\partial^IL^J\gamma^\mu\mathbf D_\mu\psi &= \gamma^\mu\mathbf D_\mu\partial^IL^J\psi +\sum_{\substack{I_1,I_2,J_1,J_2 \\ |I_2|\le |I|-1 \\ |J_2|\le |J|-1}}C^{I_1,I_2}_{J_1,J_2}(\partial^{I_1}L^{J_1}\gamma^\mu\mathbf D_\mu)(\partial^{I_2}L^{J_2}\psi) \\	
& \qquad\qquad + \sum_{\substack{ I_1,I_2,J_1,J_2 \\ |I_1|+|I_2| = |I| \\ |J_1|+|J_2| = |J|}}C^{I_1,I_2}_{J_1,J_2}(\partial^{I_1}L^{J_1}(g-m))(\partial^{I_2}L^{J_2}\partial\psi).
\end{aligned}
\end{align}
We see that an additional $t$-growth appears in the first and second summation, especially when $|J_1|\ge1$. As our previous argument concerning the commutator estimates for the Laplace-Beltrami operator $[\Box_g,L_i]$, it is not harmful. Indeed, for $|J_1|\ge1$, we see that
\begin{align*}
	& \|(\partial^{I_1}L^{J_1}\gamma^\mu\mathbf D_\mu)(\partial^{I_2}L^{J_2}\psi)\|_{L^2(\Sigma_\tau)} \\
	 & \le \|(\partial^{I_1}L^{J_1}\gamma^\mu\partial_\mu)(\partial^{I_2}L^{J_2}\psi)\|_{L^2(\Sigma_\tau)}+\|(\partial^{I_1}L^{J_1}\gamma^\mu\Gamma_\mu)(\partial^{I_2}L^{J_2}\psi)\|_{L^2(\Sigma_\tau)} \\
	& \lesssim \|t^{|J_1|}\langle r\rangle^{-2-|J_1|}\partial_t \partial^{I_2}L^{J_2}\psi\|_{L^2(\Sigma_\tau)}+\sum_{j=1}^3\|t^{|J_1|}\langle r\rangle^{-2-|J_1|}\partial_j \partial^{I_2}L^{J_2}\psi\|_{L^2(\Sigma_\tau)}.
\end{align*}
Applying the Hardy inequality Lemma \ref{hardy-ineq} $|J_1|$-times, we obtain up to an obvious error term,
\begin{align*}
	\|(\partial^{I_1}L^{J_1}\gamma^\mu\mathbf D_\mu)(\partial^{I_2}L^{J_2}\psi)\|_{L^2(\Sigma_\tau)} & \lesssim \| \langle r\rangle^{-2}\partial_t\partial^{I_2}L^{J_1}L^{J_2}\psi\|_{L^2(\Sigma_\tau)}+\sum_{j=1}^3\| \langle r\rangle^{-2}\partial_j\partial^{I_2}L^{J_1}L^{J_2}\psi\|_{L^2(\Sigma_\tau)}.
\end{align*}
The identical argument is applied to the second summation of the right-hand side of \eqref{eq-dirac-vf}, and hence we obtain
\begin{align}\label{eq-reduced}
\|	\partial^IL^J\gamma^\mu\mathbf D_\mu\psi\|_{L^2(\Sigma_\tau)} & \lesssim \|\gamma^\mu\mathbf D_\mu\partial^IL^J\psi\|_{L^2(\Sigma_\tau)}+\sum_{|I_1|\le |I|}\|\partial \partial^{I_1}L^J\psi\|_{L^2(\Sigma_\tau)}.
\end{align}
Finally, we recall that $\mathbf D_\mu= \partial_\mu-\Gamma_\mu$ and observe that the term $\|\gamma^\mu\Gamma_\mu\partial^IL^J\psi\|_{L^2(\Sigma_\tau)}$ can be absorbed into the summation of the right-hand side of the identity \eqref{eq-reduced}, which completes the proof of Proposition \ref{prop-dirac-reduce}.


\section{Main estimates: Bootstrap argument}\label{sec:main-proof}
This section is the main part of the paper. We establish the global existence of the solutions to the cubic Dirac equation and the Dirac-Klein-Gordon systems via the bootstrap argument.
\subsection{Framework of the bootstrap argument}
As an expository setup, we consider the inhomogeneous Dirac equation $(-i\gamma^\mu\mathbf D_\mu+M)\psi = F$. By Proposition \ref{prop-squaring-dirac}, we obtain the following nonlinear wave-type equation:
\begin{align*}
	\left( \Box_g-M^2\right)\psi = -2\Gamma^\mu\partial_\mu\psi+V\psi -i \gamma^\mu\mathbf D_\mu F- MF,
\end{align*}
where $V$ is the potential satisfying the regularity condition $|\partial_x^\alpha V(x)|\le c_\alpha\langle r\rangle^{-3-|\alpha|}$ for $|\alpha|\le N$, and the constants $c_\alpha$ is sufficiently small.

In what follows, we make the bootstrap assumptions: for $N\ge11$, the spinor field $\psi$ satisfies
\begin{align}\label{bootstrap-spinor}
\begin{aligned}
E_M[\partial^IL^J\psi](\tau)^\frac12 & \le C\varepsilon \tau^{\frac12+(|I'|+|J|)\delta}, \quad N-3\le |I|+|J|\le N,	\\
E_M[\partial^IL^J\psi](\tau)^\frac12 & \le C\varepsilon \tau^{(|I'|+|J|)\delta}, \quad  |I|+|J|\le N-4,
\end{aligned}
\end{align}
and for the scalar field $\phi$ we make the similar assumptions:
\begin{align}\label{bootstrap-scalar}
\begin{aligned}
E_m[\partial^IL^J\phi](\tau)^\frac12 & \le C\varepsilon \tau^{\frac12+(|I'|+|J|)\delta}, \quad N-3\le |I|+|J|\le N,	\\
E_m[\partial^IL^J\phi](\tau)^\frac12 & \le C\varepsilon \tau^{(|I'|+|J|)\delta}, \quad  |I|+|J|\le N-4,
\end{aligned}
\end{align}
where the energy $E_c[\psi]$ is defined by
\begin{align*}
	E_c[\phi](\tau) = \int_{\Sigma_{\tau}}|\partial_t\phi|^2+|\partial_x\phi|^2+2\frac{x^j}{t}\Re(\partial_j\phi\overline{\partial_t\phi})+c^2|\phi|^2\,dx,
\end{align*}
and we put $\partial^I=\partial_t^{I'}\partial_x^{I''}$. Now we apply a product of vector fields $\partial^IL^J$ to the both sides of the nonlinear equation and get
\begin{align}\label{eq-kg-vfs}
(\Box_g-M^2)\partial^IL^J\psi = [\partial^IL^J,\Box_g]\psi -\partial^IL^J\Gamma^\mu\partial_\mu\psi +\partial^IL^JV\psi -i\partial^IL^J\gamma^\mu\mathbf D_\mu F-M\partial^IL^JF	.
\end{align}
In view of the identity \eqref{eq-kg-vfs}, regardless of the nonlinear term $F$, one has to take in account the linear terms $\|\partial^IL^J\Gamma^\mu\partial_\mu\psi\|_{L^2(\Sigma_\tau)}$ and $\|\partial^IL^JV\psi\|_{L^2(\Sigma_\tau)}$ and gain at least $\tau^{-1}$ to close the bootstrap argument of nonlinear problems for Dirac equations. Indeed, we have the following estimates:
\begin{prop}\label{prop-est-linear}
    Under the bootstrap assumptions \eqref{bootstrap-spinor}, we have
    \begin{align}
\int_{\tau_0}^\tau \|\partial^IL^J\Gamma^\mu\partial_\mu\psi\|_{L^2(\Sigma_{\tau'})}\,d\tau' +\int_{\tau_0}^\tau \|\partial^IL^JV\psi\|_{L^2(\Sigma_{\tau'})}\,d\tau' \le \frac14 C\varepsilon \tau^{\frac12+(|I'|+|J|)\delta},
    \end{align}
    where $\partial^I = \partial_t^{I'}\partial_x^{I''}$ and $|I|+|J|\le N$.
\end{prop}
We postpone the proof to the end of this section.
Thus, combined with Proposition \ref{prop-error-box}, the remaining task is to control the nonlinearity.
From now on we concentrate ourselves to the cubic Dirac and the Dirac-Klein-Gordon systems, given by
\begin{align}
\begin{aligned}
	(-i\gamma^\mu\mathbf D_\mu+M)\psi = (\psi^\dagger\gamma^0\psi)\psi, \\
	\psi|_{t=t_0} := \psi_0,
	\end{aligned}
\end{align}
and
\begin{align}
\begin{aligned}
	(-i\gamma^\mu\mathbf D_\mu+M)\psi = \phi\psi, \\
	(\Box_g-m^2)\phi = \psi^\dagger\gamma^0\psi, \\
	(\psi,\phi,\partial_t\phi)|_{t=t_0}:= (\psi_0,\phi_0,\phi_1),
	\end{aligned}
\end{align}
where we set the initial time to be $t_0=2$.
By squaring the Dirac operator, Proposition \ref{prop-squaring-dirac} yields the following nonlinear wave-type equations:
\begin{align}\label{cubic-kg}
\begin{aligned}
	(\Box_g-M^2)\psi &= -2\Gamma^\mu\partial_\mu\psi +V\psi -i\gamma^\mu\mathbf D_\mu[(\psi^\dagger\gamma^0\psi)\psi]-M(\psi^\dagger\gamma^0\psi)\psi,\\
	\psi|_{t=t_0} &:= \psi_0,
\end{aligned}	
\end{align}
and
\begin{align}\label{DKG-kg}
\begin{aligned}
	(\Box_g-M^2)\psi &= -2\Gamma^\mu\partial_\mu\psi +V\psi -i\gamma^\mu\mathbf D_\mu(\phi\psi)-M\phi\psi,\\
    (\Box_g-m^2)\phi & = \psi^\dagger\gamma^0\psi , \\
(\psi,\phi,\partial_t\phi)|_{t=t_0}&:= (\psi_0,\phi_0,\phi_1).
\end{aligned}	
\end{align}
Note that the data $\partial_t\psi|_{t=t_0}$ has already been determined from the original first-order equation.

In what follows, we consider both the cubic Dirac \eqref{cubic-kg} and the Dirac-Klein-Gordon systems \eqref{DKG-kg} in a comprehensive way.
Let $\psi$ and $(\psi,\phi)$ be local-in-time solutions to the Cauchy problem associated to the cubic problem or the quadratic system, respectively. We set the initial hyperbolic time $\tau_0=2$.
A standard local analysis ensures to construct a local-in-time solution from the data given on the initial hyperboloid $\Sigma_{\tau_0}$ and for all $|I|+|J|\le N$,
\begin{align}
    E_M[\partial^IL^J\psi](\tau_0)^\frac12+E_m[\partial^IL^J\phi](\tau_0)^\frac12 \le C_0\varepsilon,
\end{align}
for some absolute constant $C_0>0$. We refer the readers to Section 11 of \cite{flochma1} for the details.
Now we make the bootstrap assumptions. For some hyperbolic time interval $[\tau_0,\tau_1]$, we suppose that the following energy inequalities hold on the interval $[\tau_0,\tau_1]$:
\begin{align*}
	E_M[\partial^IL^J\psi](\tau)^\frac12 &\le C\varepsilon \tau^{\frac12+(|I'|+|J|)\delta}, \quad N-3\le |I|+|J|\le N, \\
	E_M[\partial^IL^J\psi](\tau)^\frac12 &\le C\varepsilon \tau^{(|I'|+|J|)\delta}, \quad  |I|+|J|\le N-4, \\
    E_m[\partial^IL^J\phi](\tau)^\frac12 & \le C\varepsilon \tau^{\frac12+(|I'|+|J|)\delta}, \quad N-3\le |I|+|J|\le N,	\\
E_m[\partial^IL^J\phi](\tau)^\frac12 & \le C\varepsilon \tau^{(|I'|+|J|)\delta}, \quad  |I|+|J|\le N-4,
\end{align*}
where we put $\partial^I = \partial_t^{I'}\partial_x^{I''}$ and we fix $\frac{1}{10N}\le\delta\le \frac{1}{5N}$.
We let $\tau_*=\sup\{ \tau_1 : \textrm{ the bootstrap assumptions hold on }[\tau_0,\tau_1] \}$. By choosing $C>4C_0$, we have $\tau^*\ge2$. By Proposition \ref{prop-dirac-reduce}, in order to prove Theorem \ref{thm-global-cubic} and Theorem \ref{thm-global-dkg} via the bootstrap argument, we are only left to prove the following estimates:
\begin{prop}\label{prop-cubic-est}
    Under the bootstrap assumptions \eqref{bootstrap-spinor}, we have for $|I|+|J|\le N$,
    \begin{align}
        \sum_{|I_1|\le |I|}\int_{\tau_0}^\tau \|\partial_\mu\partial^{I_1}L^J(\psi^\dagger\gamma^0\psi)\psi\|_{L^2(\Sigma_{\tau'})}\,d\tau' \le \frac12 C\varepsilon \tau^{\frac12+(|I'|+|J|)\delta}, \quad \mu=0,1,2,3.
    \end{align}
\end{prop}
\begin{prop}\label{prop-dkg-est}
    Under the bootstrap assumptions \eqref{bootstrap-spinor}, \eqref{bootstrap-scalar}, we have for $|I|+|J|\le N$,
    \begin{align}
        \begin{aligned}
          \sum_{|I_1|\le|I|}  \int_{\tau_0}^\tau \|\partial_\mu\partial^{I_1}L^J (\phi\psi)\|_{L^2(\Sigma_{\tau'})}\,d\tau' &\le \frac12 C\varepsilon \tau^{\frac12+(|I'|+|J|)\delta}, \\
            \sum_{|I_1|\le|I|}  \int_{\tau_0}^\tau \|\partial_\mu\partial^{I_1}L^J (\psi^\dagger\gamma^0\psi)\|_{L^2(\Sigma_{\tau'})}\,d\tau' &\le \frac12 C\varepsilon \tau^{\frac12+(|I'|+|J|)\delta},
        \end{aligned}
    \end{align}
    with $\mu=0,1,2,3$.
\end{prop}
Combining with the linear estimates Proposition \ref{prop-error-box} and Proposition \ref{prop-est-linear}, the nonlinear estimates Proposition \ref{prop-cubic-est} and Proposition \ref{prop-dkg-est} imply that the bootstrap argument can be closed and $\tau^*=+\infty$. In the remainder of this paper, we focus on the proof of Proposition \ref{prop-cubic-est}, Proposition \ref{prop-dkg-est}, and Proposition \ref{prop-est-linear}.
\subsection{Proof of Proposition \ref{prop-cubic-est}}
It suffices to deal with the case $|I_1|=|I|$ in the summation. Then
we write for $\mu=0,1,2,3$,
\begin{align*}
    \int_{\tau_0}^\tau \|\partial_\mu \partial^IL^J [(\psi^\dagger\gamma^0\psi)\psi]\|_{L^2(\Sigma_{\tau'})}\,d\tau ' & \le \sum_{\substack{I_1,I_2,I_3 \\ J_1,J_2,J_3}} \int_{\tau_0}^\tau \|(\partial_\mu \partial^{I_1}L^{J_1}\psi)^\dagger\gamma^0(\partial^{I_2}L^{J_2}\psi)(\partial^{I_3}L^{J_3}\psi)\|_{L^2(\Sigma_{\tau'})}\,d\tau',
\end{align*}
where we may assume that $|I_1|+|J_1|\ge |I_2|+|J_2|\ge |I_3|+|J_3|$. Then we see that
\begin{align*}
   & \sum_{\substack{I_1,I_2,I_3 \\ J_1,J_2,J_3}} C^{I_1,I_2,I_3}_{J_1,J_2,J_3}\int_{\tau_0}^\tau \|(\partial_\mu \partial^{I_1}L^{J_1}\psi)^\dagger\gamma^0(\partial^{I_2}L^{J_2}\psi)(\partial^{I_3}L^{J_3}\psi)\|_{L^2(\Sigma_{\tau'})}\,d\tau' \\
   & \le \sum_{\substack{I_1,I_2,I_3 \\ J_1,J_2,J_3}} C^{I_1,I_2,I_3}_{J_1,J_2,J_3}\sum_{j=1}^3\int_{\tau_0}^\tau \|(\partial_j\partial^{I_1}L^{J_1}\psi)^\dagger\gamma^0(\partial^{I_2}L^{J_2}\psi)(\partial^{I_3}L^{J_3}\psi)\|_{L^2(\Sigma_{\tau'})}\,d\tau' \\
   & \qquad + \sum_{\substack{I_1,I_2,I_3 \\ J_1,J_2,J_3}} C^{I_1,I_2,I_3}_{J_1,J_2,J_3}\int_{\tau_0}^\tau \|(\partial_t \partial^{I_1}L^{J_1}\psi)^\dagger\gamma^0(\partial^{I_2}L^{J_2}\psi)(\partial^{I_3}L^{J_3}\psi)\|_{L^2(\Sigma_{\tau'})}\,d\tau'.
\end{align*}
Now we apply in order the H\"older inequality, the Klainerman-Sobolev inequality Lemma \ref{KS-ineq} on the hyperboloid and the bootstrap assumptions \eqref{bootstrap-spinor} to get
\begin{align*}
 &   \sum_{\substack{I_1,I_2,I_3 \\ J_1,J_2,J_3}} C^{I_1,I_2,I_3}_{J_1,J_2,J_3}\int_{\tau_0}^\tau \|(\partial_t \partial^{I_1}L^{J_1}\psi)^\dagger\gamma^0(\partial^{I_2}L^{J_2}\psi)(\partial^{I_3}L^{J_3}\psi)\|_{L^2(\Sigma_{\tau'})}\,d\tau' \\
    & \le C' \int_{\tau_0}^\tau \|\partial_t\partial^{I_1}L^{J_1}\psi\|_{L^2(\Sigma_{\tau'})}\|\partial^{I_2}L^{J_2}\psi\|_{L^\infty(\Sigma_{\tau'})}\|\partial^{I_3}L^{J_3}\psi\|_{L^\infty(\Sigma_{\tau'})}\,d\tau' \\
    & \le C' M^{-2} \int_{\tau_0}^\tau \frac{t}\tau C\varepsilon\tau^{\frac12+(|I_1'|+|J_1|)\delta}t^{-3} C^2\varepsilon^2 \tau^{(|I_2'|+|J_2|+|I_3'|+|J_3|+4)\delta}\,d\tau' \\
    & \le \frac12C'C^3\varepsilon^3M^{-2},
\end{align*}
where we used the fact that $\|\partial_t\psi\|_{L^2(\Sigma_\tau)}\le \frac{t}{\tau}E[\psi](\tau)^\frac12$.
For $\mu=j$, we similarly have
\begin{align*}
&    \sum_{\substack{I_1,I_2,I_3 \\ J_1,J_2,J_3}} C^{I_1,I_2,I_3}_{J_1,J_2,J_3}\sum_{j=1}^3\int_{\tau_0}^\tau \|(\partial_j\partial^{I_1}L^{J_1}\psi)^\dagger\gamma^0(\partial^{I_2}L^{J_2}\psi)(\partial^{I_3}L^{J_3}\psi)\|_{L^2(\Sigma_{\tau'})}\,d\tau' \\
& \le 3C' \int_{\tau_0}^\tau \|\partial_j\partial^{I_1}L^{J_1}\psi\|_{L^2(\Sigma_{\tau'})}\|\partial^{I_2}L^{J_2}\psi\|_{L^\infty(\Sigma_{\tau'})}\|\partial^{I_3}L^{J_3}\psi\|_{L^\infty(\Sigma_{\tau'})}\,d\tau' \\
    & \le 3C'  \int_{\tau_0}^\tau \|(\partial_j+\frac{x^j}{t}\partial_t)\partial^{I_1}L^{J_1}\psi\|_{L^2(\Sigma_{\tau'})}\|\partial^{I_2}L^{J_2}\psi\|_{L^\infty(\Sigma_{\tau'})}\|\partial^{I_3}L^{J_3}\psi\|_{L^\infty(\Sigma_{\tau'})}\,d\tau' \\
    & \qquad + 3C'  \int_{\tau_0}^\tau \|\partial_t\partial^{I_1}L^{J_1}\psi\|_{L^2(\Sigma_{\tau'})}\|\partial^{I_2}L^{J_2}\psi\|_{L^\infty(\Sigma_{\tau'})}\|\partial^{I_3}L^{J_3}\psi\|_{L^\infty(\Sigma_{\tau'})}\,d\tau' \\
    & \le 3C'C^3\varepsilon^3M^{-2},
\end{align*}
where we used the fact $\|(\partial_j+\frac{x^j}{t}\partial_t)\psi\|_{L^2(\Sigma_\tau)}\le E[\psi](\tau)^\frac12$.
We choose $\varepsilon=\frac{M}{3C'C}$ and completes the proof of Proposition \ref{prop-cubic-est}.
\subsection{Proof of Proposition \ref{prop-dkg-est}}
As the previous proof, it is enough to consider the case $|I_1|=|I|$ in the summation. Then
we have
\begin{align*}
    \int_{\tau_0}^\tau \|\partial_\mu \partial^IL^J(\phi\psi)\|_{L^2(\Sigma_{\tau'})}\,d\tau' & \le \sum_{\substack{I_1,J_1 \\ I_2, J_2}}\int_{\tau_0}^\tau \|(\partial_\mu\partial^{I_1}L^{J_2}\phi )(\partial^{I_2}L^{J_2}\psi)\|_{L^2(\Sigma_{\tau'})}\,d\tau',
\end{align*}
where we may assume that $|I_1|+|J_1| \ge |I_2|+|J_2|$ and $\mu=0,1,2,3$. Then
\begin{align*}
  &  \sum_{\substack{I_1,J_1 \\ I_2, J_2}}\int_{\tau_0}^\tau \|(\partial_\mu\partial^{I_1}L^{J_2}\phi )(\partial^{I_2}L^{J_2}\psi)\|_{L^2(\Sigma_{\tau'})}\,d\tau' \\
  & \le \sum_{\substack{I_1,J_1 \\ I_2, J_2}}C^{I_1,J_1}_{I_2,J_2}\int_{\tau_0}^\tau \|(\partial_t\partial^{I_1}L^{J_2}\phi )(\partial^{I_2}L^{J_2}\psi)\|_{L^2(\Sigma_{\tau'})}\,d\tau' \\
  & \qquad + \sum_{\substack{I_1,J_1 \\ I_2, J_2}}C^{I_1,J_1}_{I_2,J_2}\sum_{j=1}^3\int_{\tau_0}^\tau \|(\partial_j\partial^{I_1}L^{J_2}\phi )(\partial^{I_2}L^{J_2}\psi)\|_{L^2(\Sigma_{\tau'})}\,d\tau'.
\end{align*}
Applying the H\"older inequality, the Klainerman-Sobolev inequality, and then the bootstrap assumptions \eqref{bootstrap-spinor} and \eqref{bootstrap-scalar}, we see that
\begin{align*}
    & \sum_{\substack{I_1,J_1 \\ I_2, J_2}}C^{I_1,J_1}_{I_2,J_2}\int_{\tau_0}^\tau \|(\partial_t\partial^{I_1}L^{J_2}\phi )(\partial^{I_2}L^{J_2}\psi)\|_{L^2(\Sigma_{\tau'})}\,d\tau' \\
    & \le C' \int_{\tau_0}^\tau \|\partial_t\partial^{I_1}L^{J_1}\phi\|_{L^2(\Sigma_{\tau'})}\|\partial^{I_2}L^{J_2}\psi\|_{L^\infty(\Sigma_{\tau'})}\,d\tau' \\
    & \le C' M^{-1} \int_{\tau_0}^\tau \frac t\tau C\varepsilon \tau^{\frac12+(|I_1'|+|J_1|)\delta}t^{-\frac32}C\varepsilon \tau^{(|I_2'|+|J_2|+2)\delta}\,d\tau' \\
    & \le C'C^2\varepsilon^2 M^{-1}\int_{\tau_0}^\tau t^{-\frac12}\tau^{-\frac12+(N+2)\delta}\,d\tau' \\
    & \le C'C^2\varepsilon^2 M^{-1}\tau^{(N+2)\delta},
\end{align*}
and
\begin{align*}
   & \sum_{\substack{I_1,J_1 \\ I_2, J_2}}C^{I_1,J_1}_{I_2,J_2}\sum_{j=1}^3\int_{\tau_0}^\tau \|(\partial_j\partial^{I_1}L^{J_2}\phi )(\partial^{I_2}L^{J_2}\psi)\|_{L^2(\Sigma_{\tau'})}\,d\tau' \\
   & \le 3C' \int_{\tau_0}^\tau \|(\partial_j+\frac{x^j}t\partial_t)\partial^{I_1}L^{J_1}\phi\|_{L^2(\Sigma_{\tau'})}\|\partial^{I_2}L^{J_2}\psi\|_{L^\infty(\Sigma_{\tau'})}\,d\tau' \\
   & \qquad +3C'  \int_{\tau_0}^\tau \|\partial_t\partial^{I_1}L^{J_1}\phi\|_{L^2(\Sigma_{\tau'})}\|\partial^{I_2}L^{J_2}\psi\|_{L^\infty(\Sigma_{\tau'})}\,d\tau' \\
   & \le 6C' C^2\varepsilon^2M^{-1}\tau^{(N+2)\delta}.
\end{align*}
Thus we choose $\varepsilon=\dfrac{\min(m,M)}{12C'C}$. Note that we have not used any specific structure of the quadratic term $\phi\psi$, and hence our argument can be readily applied to the quadratic term $\psi^\dagger\gamma^0\psi$. We omit the details, and this completes the proof. 
\subsection{Proof of Proposition \ref{prop-est-linear}} 
We give the proof of the estimates for the linear terms which appear in the right-handside of the nonlinear wave-type equation derived from Proposition \ref{prop-squaring-dirac}. 
This can be achieved as follows: we first write $\Gamma^\mu\partial_\mu\psi = \Gamma^0\partial_t\psi+\Gamma^j\partial_j\psi$. Then
\begin{align*}
    \|\partial^IL^J\Gamma^j\partial_j\psi\|_{L^2(\Sigma_\tau)} & \le c\|\langle r\rangle^{-2}\partial_j\partial^IL^J\psi\|_{L^2(\Sigma_\tau)}+\sum_{|I_2|+|J_2|\le N-1}c_{I_1,J_1,I_2,J_2}\|\partial^{I_1}L^{J_2}\Gamma^j \partial_j \partial^{I_2}L^{J_2}\psi\|_{L^2(\Sigma_\tau)}.
\end{align*}
The second term turns out to be an acceptable error via the previous argument in Section 4. To control the first term, we write
\begin{align*}
    \|\langle r\rangle^{-2}\partial_j\partial^IL^J\psi\|_{L^2(\Sigma_\tau)} & \le \|(\partial_j+\frac{x^j}t\partial_t)\partial^IL^J\psi\|_{L^2(\Sigma_\tau)}+ \|\langle r\rangle^{-2}\frac{x^j}t\partial_t \partial^IL^J\psi\|_{L^2(\Sigma_\tau)}.
\end{align*}
Now we see that
\begin{align*}
     \|\langle r\rangle^{-2}\frac{x^j}t\partial_t \partial^IL^J\psi\|_{L^2(\Sigma_\tau)} & \le \tau^{-1}\|\langle r\rangle^{-1}\frac\tau t \partial_t\partial^IL^J\psi\|_{L^2(\Sigma_\tau)} \\
     & \le \tau^{-1}E[\partial^IL^J\psi](\tau)^\frac12
\end{align*}
and
\begin{align*}
    \|(\partial_j+\frac{x^j}t\partial_t)\partial^IL^J\psi\|_{L^2(\Sigma_\tau)} & \le t^{-1}\|\partial^IL^{J+1}\psi\|_{L^2(\Sigma_\tau)} \\
    & = t^{-1}\|\partial_t\partial^{I-1}L^{J+1}\psi\|_{L^2(\Sigma_\tau)} \\
    & \le \tau^{-1}E[\partial^{I-1}L^{J+1}\psi](\tau)^\frac12.
\end{align*}
If $\partial^I=\partial_x^I$, then we get a better bound. Indeed, we see that
\begin{align*}
    \|(\partial_j+\frac{x^j}t\partial_t)\partial^IL^J\psi\|_{L^2(\Sigma_\tau)} & \le\|t^{-1}\partial^IL^{J+1}\psi\|_{L^2(\Sigma_\tau)} \\
    & \le \|t^{-1}(\partial_j+\frac{x^j}{t}\partial_t)\partial_x^{I-1}L^J\psi\|_{L^2(\Sigma_\tau)}+ \|t^{-1}\frac{x^j}{t}\partial_t\partial_x^{I-1}L^J\psi\|_{L^2(\Sigma_\tau)} \\
    & \le \|t^{-2}\partial_x^{I-1}L^{J+1}\psi\|_{L^2(\Sigma_\tau)}+\|t^{-2}|x| \partial_x^{I-1}L^{J}\psi\|_{L^2(\Sigma_\tau)}.
\end{align*}
For the $\Gamma^0\partial_t\psi$, we follow the previous argument in Section \ref{sec:commutators}. Indeed, we see that
\begin{align*}
	\|\Gamma^0\partial_t\partial^IL^J\psi\|_{L^2(\Sigma_\tau)} & \lesssim \|\langle r\rangle^{-3}\partial_t\partial^IL^J\psi\|_{L^2(\Sigma_\tau)} \\
	& = \|\langle r\rangle^{-3}\partial_t\partial_i\partial_j\partial^{I-2}L^J\psi\|_{L^2(\Sigma_\tau)} \\
	& \lesssim \|t^{-2}\langle r\rangle^{-3}\partial_t \partial^{I-2}L^{J+2}\psi\|_{L^2(\Sigma_\tau)}+\|t^{-2}\langle r\rangle^{-2}\partial_t^2\partial^{I-2}L^{J+1}\psi\|_{L^2(\Sigma_\tau)} \\
	&\qquad +\|t^{-2}\langle r\rangle^{-1}\partial_t^3\partial^{I-2}L^J\psi\|_{L^2(\Sigma_\tau)}.
\end{align*}
For the $V\psi$, we simply use the Hardy-type inequality Lemma \ref{hardy-ineq} on the hyperboloid to get
\begin{align*}
    \|V\partial^IL^J\psi\|_{L^2(\Sigma_\tau)} & \le c\|\langle r\rangle^{-3}\partial^IL^J\psi\|_{L^2(\Sigma_\tau)} \\
    & \le cC'\sum_{j=1}^3\|\langle r\rangle^{-2}(\partial_j+\frac{x^j}t\partial_t)\partial^IL^J\psi\|_{L^2(\Sigma_\tau)} \\
    & \le 3cC' t^{-1}\| \partial^I L^{J+1}\psi\|_{L^2(\Sigma_\tau)} \\
    & = 3cC' \tau^{-1}\|\frac\tau t \partial_t \partial^{I-1} L^{J+1}\psi\|_{L^2(\Sigma_\tau)} \\
    & \le 3cC' \tau^{-1}E[\partial^{I-1}L^{J+1}\psi](\tau)^\frac12,
\end{align*}
where the smallness of the constant $c>0$ ensures to close the bootstrap argument provided that the underlying curved spacetime is sufficently close to the Minkowski spacetime. If $\partial^I=\partial_x^I$, we obtain a better bound via the identical manner as in the control of $\Gamma^\mu\partial_\mu\psi$. We omit the details. Thus we conclude that
\begin{align*}
    \int_{\tau_0}^\tau \|\partial^IL^J\Gamma^\mu\partial_\mu\psi\|_{L^2(\Sigma_{\tau'})}+\|\partial^IL^JV\psi\|_{L^2(\Sigma_{\tau'})}\,d\tau' \le \frac14 C\varepsilon \tau^{\frac12+(|I'|+|J|)\delta}, \quad |I|+|J|\le N, 
\end{align*}
which completes the proof. Combining all of the results from Proposition \ref{prop-est-linear}, Proposition \ref{prop-cubic-est}, and Proposition \ref{prop-dkg-est}, the bootstrap argument can be closed and $\tau^*=+\infty$, and hence the Cauchy problems of the cubic Dirac \eqref{cubic-kg} and the Dirac-Klein-Gordon systems \eqref{DKG-kg} admit global solutions. Furthermore, repeating the bootstrap argument with modified assumptions, given the data on the initial hyperboloid $\Sigma_{\tau_0}$ satisfying the energy bound,
\begin{align*}
	E_M[\partial^IL^J\psi](\tau_0)^\frac12+E_m[\partial^IL^J\phi](\tau_0)^\frac12 \le C_0\varepsilon,
\end{align*}
we can establish the global solutions $\psi$ and $(\psi,\phi)$ to \eqref{cubic-kg} and \eqref{DKG-kg}, respectively, satisfying the energy bound
\begin{align*}
	\sup_{\tau_0\le\tau}E_M [ \partial^IL^J\psi ](\tau)^\frac12+E_m[\partial^IL^J\phi](\tau)^\frac12 \le 2C_0\varepsilon.
\end{align*}
 From the bootstrap argument with the Klainerman-Sobolev inequality, we deduce that the solutions $\psi$ to \eqref{cubic-kg} and $(\psi,\phi)$ to \eqref{DKG-kg} satisfy the pointwise decay: for all $\tau\ge\tau_0$
\begin{align*}
	\sup_{(t,x)\in\Sigma_\tau}t^\frac32 |\psi(t,x)| \le C_0\varepsilon,
\end{align*}
and
\begin{align*}
	\sup_{(t,x)\in\Sigma_\tau}t^\frac32 |\psi(t,x)|+ \sup_{(t,x)\in\Sigma_\tau}t^\frac32 |\phi(t,x)| \le C_0\varepsilon,
\end{align*}
which completes the proof of Theorem \ref{thm-global-cubic} and Theorem \ref{thm-global-dkg}.
\section*{Appendix}
The purpose of this section is to present backgrounds of the hyperboloidal foliation method and establish global existence of cubic and quadratic nonlinear Klein-Gordon equations as toy models.
\subsection{Energy estimates on the hyperboloid}
We introduce the foliation $\Sigma_\tau=\{(t,x): \tau = \sqrt{t^2-|x|^2}\}$. 
This is the level set of the function $f(t,x)=\sqrt{t^2-|x|^2}$ and hence the normal vector $N_\mu$ is given by $N_\mu=\partial_\mu f=\frac1{\sqrt{t^2-|x|^2}}(t,-x^j)$ and from $N^\mu= g^{\mu\nu}N_\nu$ we have
\begin{align*}
	N^0 = g^{0\nu} \partial_{\nu}f = \frac1{\sqrt{t^2-|x|^2}}(g^{00}t-g^{0j}x^j), \quad N^j = g^{j\nu}\partial_\nu f = \frac1{\sqrt{t^2-|x|^2}}(g^{0j}t-g^{ij}x^i),
\end{align*}
or equivalently, we have 
\begin{align*}
N^\mu = \frac1{\sqrt{t^2-|x|^2}} 	(g^{\mu0}t-g^{\mu i}x^i).
 \end{align*}
Then the quantity $g^{\mu\nu}N_\mu N_\nu$ is given by
\begin{align*}
	g^{\mu\nu}N_\mu N_\nu = \frac1{t^2-|x|^2}(g^{00}t^2+g^{ij}x^ix^j-2g^{0j}tx^j) := \sigma.
\end{align*}
We also observe that
\begin{align*}
	|\sigma| = 1+\frac1{t^2-|x|^2}( 2g^{0j}tx^j-(g^{ij}-\delta^{ij})x^ix^j ).
\end{align*}
Then the unit normal vector $n^\mu$ is given by
\begin{align*}
	n^\mu = \frac1{\sqrt{|\sigma|}}\frac1{\sqrt{t^2-|x|^2}}(g^{\mu0}t-g^{\mu i}x^i).
\end{align*}
Since $g^{00}=-1$, we choose $n^\mu\to -n^\mu$ so that $n^0$ is positive and $n^\mu$ is future-directied unit time-like vector. From now on, we consider the unit normal vector
\begin{align*}
	n^\mu = \frac1{\sqrt{|\sigma|}}\frac1{\sqrt{t^2-|x|^2}}(-g^{\mu0}t+g^{\mu i}x^i).
\end{align*}
From this we define the vector field $N\phi = n^\mu\partial_\mu\phi$. 

For $\tau_0\le \tau_1$, we consider the subdomain bounded by the hyperboloids $\Sigma_{\tau_0}$ and $\Sigma_{\tau_1}$, inside the forward light cone, which is denoted by
\begin{align*}
    D_{[\tau_0,\tau_1]} = \bigcup_{\tau_0\le\tau\le\tau_1}\Sigma_{\tau} \,\cap \{(t,x): |x|<t-1\}. 
\end{align*}

We shall establish the energy estimates inside the forward light cone $\{(t,x):|x|<t-1\}$, using the hyperboloid foliation $\Sigma_\tau$. To do this, we use the contraction with the energy momentum tensor $T_{\mu\nu}$ given by
\begin{align*}
T_{\mu\nu} & = \Re(\partial_\mu\phi\overline{\partial_\nu\phi})-\frac12 g_{\mu\nu}(g^{\alpha\beta}\partial_\alpha\phi\overline{\partial_\beta\phi}+|\phi|^2),    
\end{align*}
with the vector field $X=\partial_t$. Then we use the divergence theorem for $T_{\mu\nu}X^\nu$ on the domain $\bigcup_{\tau_0\le \tau'\le \tau}\Sigma_{\tau'}$, which gives
\begin{align*}
    \int_{\Sigma_{\tau}}T_{\mu\nu}X^\nu n^\mu \,d\sigma_{\Sigma_{\tau}}-\int_{\Sigma_{\tau_0}}T_{\mu\nu}X^\nu n^\mu\,d\sigma_{\Sigma_{\tau_0}} = \int_{\tau_0}^{\tau}T^{\mu\nu}\pi_{\mu\nu}^{(X)}d\sigma_{\Sigma_{\tau'}}d\tau'+\int_{\tau_0}^{\tau}\Re(F\cdot \overline{\partial_t\phi})\,d\sigma_{\Sigma_{\tau'}}d\tau',
\end{align*}
where $\pi^{(X)}$ is the deformation tensor given by
\begin{align*}
    \pi^{(X)}_{\mu\nu} & = \nabla_\mu X_\nu +\nabla_\nu X_\mu.
\end{align*}
We note that the deformation tensor $\pi^{(X)}$ vanishes provided that $X$ is a Killing field. Here we are concerned with a non-stationary case and assume \eqref{nonstationary} and adapt the normalised coordinates.
We compute the energy flux $T_{\mu\nu}X^\nu n^\mu$:
\begin{align*}
    T_{\mu\nu}X^\nu n^\mu & = \Re(n^\mu\partial_\mu\phi \overline{X^\nu\partial_\nu\phi})-\frac12g_{\mu\nu}X^\nu n^\mu(g^{\alpha\beta}\partial_\alpha\phi\overline{\partial_\beta\phi}+|\phi|^2).
\end{align*}
Here we compute
\begin{align*}
    g_{\mu\nu}X^\nu n^\mu &= g_{00}X^0 n^0+g_{j0}X^0 n^j \\
    & = -n^0+g_{j0}n^j \\
    & = \frac1{\sqrt{|\sigma|}\sqrt{t^2-|x|^2}}(-t-g^{0i}x^i)+ \frac1{\sqrt{|\sigma|}\sqrt{t^2-|x|^2}}g_{j0}(-g^{j0}t+g^{ji}x^i) \\
    & = \frac1{\sqrt{|\sigma|}\sqrt{t^2-|x|^2}}(-t)+(Error),
\end{align*}
and
\begin{align*}
    N\phi &=  \frac1{\sqrt{|\sigma|}\sqrt{t^2-|x|^2}}\left( (t+g^{0i}x^i )\partial_t\phi+(-g^{j0}t+g^{ji}x^i )\partial_j\phi\right) \\
    & = \frac1{\sqrt{|\sigma|}\sqrt{t^2-|x|^2}}(t\partial_t\phi+x^j\partial_j\phi) \\
    & \qquad+\frac1{\sqrt{|\sigma|}\sqrt{t^2-|x|^2}}\left( (g^{0i}x^i\partial_t\phi+(-g^{j0}t+(g^{ij}-\delta^{ij}x^i)\partial_j\phi) \right) \\
    & = \frac1{\sqrt{|\sigma|}\sqrt{t^2-|x|^2}}(t\partial_t\phi+x^j\partial_j\phi)+(Error).
\end{align*}
In consequence, we have
\begin{align*}
    T_{\mu\nu}X^\nu n^\mu & = \frac{t}{\sqrt{|\sigma|}\sqrt{t^2-|x|^2}}\left( |\partial_t\phi|^2+|\partial_x\phi|^2+|\phi|^2+2\frac{x^j}{t}\Re(\partial_t\phi\overline{\partial_j\phi}) \right)+(Error),
\end{align*}
where
\begin{align*}
    (Error) &= \frac1{\sqrt{|\sigma|}\sqrt{t^2-|x|^2}}\left( (g^{ij}-\delta^{ij})(\partial_i\phi\overline{\partial_j\phi}+\partial_j\phi\overline{\partial_i\phi})+g^{j0}(\partial_t\phi\overline{\partial_j\phi}+\partial_j\phi\overline{\partial_t\phi}) \right) \\
    & \qquad +\frac{g^{0i}x^i+g_{j0}(g^{j0}t-g^{ji}x^i)}{2\sqrt{|\sigma|}\sqrt{t^2-|x|^2}}(g^{\alpha\beta}\partial_\alpha\phi\overline{\partial_\beta\phi}+|\phi|^2) \\
    &\qquad +\frac1{\sqrt{|\sigma|}\sqrt{t^2-|x|^2}}\left( (g^{0i}x^i)|\partial_t\phi|^2+(-g^{j0}t+(g^{ij}-\delta^{ij})x^i)\Re(\partial_j\phi\overline{\partial_t\phi}) \right).
\end{align*}
The factor $t$, which appears in the error terms, should not be harmful. Indeed, we observe that
\begin{align*}
	t &= t-\sqrt{t^2-|x|^2}+\sqrt{t^2-|x|^2} \\
	& = \frac{t^2-(t^2-|x|^2)}{t+\sqrt{t^2-|x|^2}}+\sqrt{t^2-|x|^2} \\
	& = \frac{|x|^2}{t+\sqrt{t^2-|x|^2}}+\sqrt{t^2-|x|^2} \le \frac{|x|^2}{t}+\sqrt{t^2-|x|^2}.
\end{align*}
Then we see that
\begin{align*}
	\frac{|g^{j0}|t}{\sqrt{t^2-|x|^2}} \le c\left( \frac1{t\sqrt{t^2-|x|^2}}+\langle x\rangle^{-2} \right),
\end{align*}
where $c$ is a sufficiently small constant, due to the non-trapping condition. Thus we conclude that
\begin{align*}
	|(Error)| \le c|\partial\phi|^2,
\end{align*}
and this can be absorbed into the leading terms of the energy flux. Indeed, we define the energy
\begin{align}
    E^{}[\phi](\tau) = \int_{\Sigma_\tau}\left( |\partial_t\phi|^2+|\partial_x\phi|^2+|\phi|^2+2\frac{x^j}{t}\Re(\partial_t\phi\overline{\partial_j\phi}) \right)\,dx
\end{align}
or equivalently
\begin{align}
E[\phi](\tau) = 	\int_{\Sigma_\tau} \left( \sum_{j=1}\left| \partial_j\phi+\frac{x^j}{t}\partial_t\phi \right|^2+\left|\frac{\tau}{t}\partial_t\phi\right|^2 +|\phi|^2\right)\,dx,
\end{align}
where we used the fact that the surface measure satisfies the identity $d\sigma_{\Sigma_\tau}=\frac{\tau}{t}dx$. In order to establish the energy inequality, we need to control the bulk term $\int T\pi$.
For $X=\partial_t$, we have $X^0=1$ and $X^j=0$, $j=1,2,3$. Then $X_\nu=g_{\nu0}X^0=g_{\nu0}$ and hence
\begin{align*}
	\pi_{00} & = 2(\partial_tg_{00}-\Gamma^\lambda_{00}g_{\lambda0}) = -2\Gamma^\lambda_{00}g_{\lambda0}, \\
	\pi_{0i} & = \partial_tg_{i0}-2\Gamma^\lambda_{i0}g_{\lambda0} ,\\
	\pi_{ij} & = \partial_ig_{j0}+\partial_jg_{i0}-2\Gamma^\lambda_{ij}g_{\lambda0}.
\end{align*}
We recall that the Christoffel symbols are given by
\begin{align*}
	\Gamma^\lambda_{\mu\nu} &= \frac12 g^{\lambda\rho}(\partial_\mu g_{\nu\rho}+\partial_\nu g_{\mu\rho}-\partial_\rho g_{\mu\nu}),
\end{align*}
and we compute
	\begin{align*}
		|\Gamma^0_{00} |& = \frac12 |g^{0\rho}(\partial_tg_{0\rho}+\partial_tg_{0\rho}-\partial_\rho g_{00}) |\le c^2 \langle r+t\rangle^{-2-2\eta}\langle r\rangle^{-3+2\eta}, \\
		|\Gamma^i_{00}| & = \frac12 |g^{i\rho}(\partial_tg_{0\rho}+\partial_tg_{0\rho}-\partial_\rho g_{00}) |\le c \langle r+t\rangle^{-1-\eta}\langle r\rangle^{-2+\eta}, \\
		|\Gamma^0_{i0}| &= \frac12 |g^{0\rho}(\partial_ig_{0\rho}+\partial_tg_{i\rho}-\partial_\rho g_{i0})| \le c \langle r+t\rangle^{-1-\eta}\langle r\rangle^{-2+\eta} ,\\
		|\Gamma^j_{i0}| & = \frac12| g^{j\rho}(\partial_ig_{0\rho}+\partial_tg_{i\rho}-\partial_\rho g_{i0})| \le c \langle r+t\rangle^{-1-\eta}\langle r\rangle^{-2+\eta} , \\
		|\Gamma^0_{ij}| &= \frac12 |g^{0\rho}(\partial_ig_{j\rho}+\partial_jg_{i\rho}-\partial_\rho g_{ij})| \le c \langle r+t\rangle^{-1-\eta}\langle r\rangle^{-2+\eta}, \\
		|\Gamma^k_{ij}| & = \frac12 |g^{k\rho}(\partial_ig_{j\rho}+\partial_jg_{i\rho}-\partial_\rho g_{ij})| \le c\langle r\rangle^{-3}+c^2 \langle r+t\rangle^{-1-\eta}\langle r\rangle^{-2+\eta}  ,
		\end{align*}
with a sufficiently small constant $c>0$. The $\Gamma^k_{ij}$ do not reveal a decay in $|r+t|$, which seems problematic. However, this term only appears in $\pi_{ij}$ and $\Gamma^\lambda_{ij}g_{\lambda0}=\Gamma^0_{ij}g_{00}+\Gamma^k_{ij}g_{k0}$. Thus $g_{k0}$ yields the decay in $|r+t|$. Therefore, we deduce that
\begin{align*}
    |\pi_{\mu\nu}^{(X)}| \le c\langle r+t\rangle^{-1-\eta}\langle r\rangle^{-2+\eta}.
\end{align*}
Now we compute the components of the energy-momentum tensor:
	  \begin{align*}
	  	T^{00} & = \partial^0\phi\partial^0\phi-\frac12g^{00}(-|\partial_t\phi|^2+|\partial_x\phi|^2+|\phi|^2+(g^{ij}-\delta^{ij})\partial_i\phi\partial_j\phi+2g^{0j}\partial_t\phi\partial_j\phi) \\
	  	& = g^{\mu0}g^{\nu0}\partial_\mu\phi\partial_\nu\phi+\frac12 (-|\partial_t\phi|^2+|\partial_x\phi|^2+|\phi|^2+(g^{ij}-\delta^{ij})\partial_i\phi\partial_j\phi+2g^{0j}\partial_t\phi\partial_j\phi)\\
	  	& = |\partial_t\phi|^2+g^{i0}g^{j0}\partial_i\phi\partial_j\phi +\frac12 (-|\partial_t\phi|^2+|\partial_x\phi|^2+|\phi|^2+(g^{ij}-\delta^{ij})\partial_i\phi\partial_j\phi+2g^{0j}\partial_t\phi\partial_j\phi) \\
	  	& = \frac12(|\partial_t\phi|^2+|\partial_x\phi|^2+|\phi|^2) + (Error),
	  \end{align*}
	  and
	  \begin{align*}
	  	T^{0j} & = \partial^0\phi\partial^j\phi-\frac12 g^{0j}(-|\partial_t\phi|^2+|\partial_x\phi|^2+|\phi|^2+(g^{ij}-\delta^{ij})\partial_i\phi\partial_j\phi+2g^{0j}\partial_t\phi\partial_j\phi) \\
	  	& =g^{\mu0}g^{\nu j}\partial_\mu\phi \partial_\nu\phi -\frac12 g^{0j}(-|\partial_t\phi|^2+|\partial_x\phi|^2+|\phi|^2+(g^{ij}-\delta^{ij})\partial_i\phi\partial_j\phi+2g^{0j}\partial_t\phi\partial_j\phi) \\
	  	& = -\partial_t\phi\partial_j\phi + (Error),
	  \end{align*}
	  and
	  \begin{align*}
	  	T^{ij} & = \partial_i\phi\partial_j\phi-\frac12\delta^{ij}(-|\partial_t\phi|^2+|\partial_x\phi|^2+|\phi|^2)+(Error).
	  \end{align*}
In consequence, we see that
\begin{align*}
    \left| \int_{D_{[\tau_0,\tau]}}T^{\mu\nu}\pi_{\mu\nu}^{(X)} \right| & \le c\int_{\tau_0}^{\tau} \frac1{\tau^{1+\eta}}E[\phi](\tau')\,d\tau' \le c\sup_{\tau_0\le\tau'\le\tau}E[\phi](\tau'),
\end{align*}
for some sufficiently small $c>0$.
Then
we may choose constants $c_1<c_2$ so that
\begin{align*}
	c_1E[\phi](\tau) & \le c_2E[\phi](\tau_0) +c\sup_{\tau_0\le\tau'\le\tau}E[\phi](\tau')+\int_{\tau_0}^\tau \int_{\Sigma_\tau}|F\cdot\partial_t\phi|\,d\sigma_{\Sigma_\tau}\,d\tau' \\
    & = c_2E^{}[\phi](\tau_0)+c\sup_{\tau_0\le\tau'\le\tau}E[\phi](\tau')+\int_{\tau_0}^\tau \int_{\Sigma_\tau}|F||\partial_t\phi|\frac{\tau}{t}\, dx d\tau' \\
    & \le  c_2E^{}[\phi](\tau_0)+c\sup_{\tau_0\le\tau'\le\tau}E[\phi](\tau')+\int_{\tau_0}^\tau \|F\|_{L^2(\Sigma_{\tau'})}\,d\tau' \sup_{\tau_0\le \tau'\le \tau}E^{}[\phi](\tau')^\frac12.
\end{align*}
Then the standard energy argument implies
\begin{align*}
    \sup_{\tau_0\le \tau'\le \tau}E^{}[\phi](\tau')^\frac12 \lesssim E^{}[\phi](\tau_0)^\frac12+\int_{\tau_0}^\tau \|F\|_{L^2(\Sigma_{\tau'})}\,d\tau',
\end{align*}
where the bulk term is absorbed into the left-hand side due to the smallness of $c>0$.
We summarise it as follows:
\begin{prop}
For the solution $\phi$ to the inhomogeneous Klein-Gordon equation $(\Box_g-1)\phi =F$, which are spatially supported inside the forward light cone $\{(t,x): |x|<t-1\}$, we have the energy estimates:
    \begin{align}
         \sup_{\tau_0\le \tau'\le \tau}E^{}[\phi](\tau')^\frac12 \lesssim E^{}[\phi](\tau_0)^\frac12+\int_{\tau_0}^\tau \|F\|_{L^2(\Sigma_{\tau'})}\,d\tau'.
    \end{align}
\end{prop}
\subsection{Toy model: Cubic Klein-Gordon equations}
We consider the cubic Klein-Gordon equation 
\begin{align*}
	(\Box_g-1)\phi = |\phi|^2\phi, \quad (t,x)\in \{ (t,x): t\ge2, |x|\le t-1 \}.
\end{align*}
with initial data $\phi|_{t=2}=\phi_0,\ \partial_t\phi|_{t=2}=\phi_1$. 
The goal is to establish global solution of the cubic problem, given any spatially compactly supported, smooth initial data with $\|\phi_0\|_{H^N}+||\phi_1\|_{H^{N-1}}\le\varepsilon$. 

Let $\phi$ be the local-in-time solution to the Cauchy problem associated to the cubic problem. A standard local analysis ensures to construct a local-in-time solution from the data given on the initial hyperboloid $\Sigma_{\tau_0}$ with $\tau_0=2$, and for all $|I|+|J|\le N$,
\begin{align}
    E[\partial^IL^J\phi](\tau_0)^\frac12 \le C_0\varepsilon,
\end{align}
for some absolute constant $C_0>0$.
Now we make the bootstrap assumptions: for some hyperbolic time interval $[\tau_0,\tau_1]$, the following inequalities hold:
\begin{align*}
	E[\partial^IL^J\phi](\tau)^\frac12 \le C\varepsilon \tau^{(|I'|+|J|)\delta}, \quad |I|+|J|\le N-4, \\
	E[\partial^IL^J\phi](\tau)^\frac12 \le C\varepsilon \tau^{\frac12+(|I'|+|J|)\delta}, \quad N-3\le |I|+|J|\le N,
\end{align*}
where we put $\partial^I = \partial_t^{I'}\partial_x^{I''}$ and we fix $\frac{1}{10N}\le\delta\le \frac{1}{5N}$.
We let $\tau_*=\sup\{ \tau_1 : \textrm{ the bootstrap assumptions hold on }[\tau_0,\tau_1] \}$. By choosing $C>C_0$, we have $\tau^*\ge2$.
Applying the energy estimate Proposition \ref{energy-ineq}, we have
\begin{align*}
	\sup_{\tau}E[\partial^IL^J\phi][\tau]^\frac12 \le C'E[\partial^IL^J\phi](\tau_0)^\frac12+\int_{\tau_0}^{\tau} \|[\partial^IL^J,\Box_g]\phi\|_{L^2(\Sigma_{\tau'})}\,d\tau' +\int_{\tau_0}^\tau \|\partial^IL^J(|\phi|^2\phi)\|_{L^2(\Sigma_{\tau'})}\,d\tau'.
\end{align*}
Then Proposition \ref{prop-error-box} gives
\begin{align*}
	\int_{\tau_0}^{\tau} \|[\partial^IL^J,\Box_g]\phi\|_{L^2(\Sigma_{\tau'})}\,d\tau'  & \le \sum_{|I'|+|J'|\le N}c_{I',J'}\int_{\tau_0}^{\tau} t^{-2}\|\partial^{I'}L^{J'}\phi\|_{L^2(\Sigma_{\tau'})}\,d\tau' \\
	& \le C\varepsilon\int_{\tau_0}^{\tau} t^{-2}(\tau')^{\frac12+(N+2)\delta}\,d\tau' \\
	& \le \frac12 C\varepsilon.
\end{align*}
For the cubic term, we see that
\begin{align*}
	\int_{\tau_0}^{\tau} \|\partial^IL^J(|\phi|^2\phi)\|_{L^2(\Sigma_{\tau'})}\,d\tau' & \le \sum_{\substack{I_1,I_2,I_3 \\ J_1,J_2, J_3}} c^{I_1,I_2,I_3}_{J_1,J_2,J_3}\int_{\tau_0}^{\tau}\|\partial^{I_1}L^{J_1}\phi \cdot \partial^{I_2}L^{J_2}\phi\cdot \partial^{I_3}L^{J_3}\phi\|_{L^2(\Sigma_{\tau'})}\,d\tau'.
\end{align*}
We may assume that $|I_1|+|J_1| \ge |I_2|+ |J_2| \ge |I_3|+ |J_3|$. Then using the Klainerman-Sobolev-type estimates Lemma \ref{KS-ineq} we have
\begin{align*}
&	\sum_{\substack{I_1,I_2,I_3 \\ J_1,J_2, J_3}}c^{I_1,I_2,I_3}_{J_1,J_2,J_3}\int_{\tau_0}^\infty \|\partial^{I_1}L^{J_1}\phi \cdot \partial^{I_2}L^{J_2}\phi\cdot \partial^{I_3}L^{J_3}\phi\|_{L^2(\Sigma_\tau)}\,d\tau \\
	& \le \sum_{\substack{I_1,I_2,I_3 \\ J_1,J_2, J_3}}c^{I_1,I_2,I_3}_{J_1,J_2,J_3} \int_{\tau_0}^{\tau}\|\partial^{I_1}L^{J_1}\phi \|_{L^2(\Sigma_{\tau'})} \| \partial^{I_2}L^{J_2}\phi\|_{L^\infty(\Sigma_{\tau'})} \|\partial^{I_3}L^{J_3}\phi\|_{L^\infty(\Sigma_{\tau'})}\,d\tau' \\
	& \le C_1 C^3\varepsilon^3 \int_{\tau_0}^{\tau} (\tau')^{\frac12+(|I_1'|+|J_1|)\delta}t^{-3}(\tau')^{(|I_2'|+|J_2|+|I_3'|+|J_3|+4)\delta}\,d\tau' \\
	&\le C_1C^3\varepsilon^3\int_{\tau_0}^{\tau} t^{-3} (\tau')^{\frac12+(N+4)\delta}\,d\tau' \\
    & \le \frac14 C_1C^3\varepsilon^3 ,
\end{align*}
where we used the fact that if $|I_1|+|J_1|\ge N-3$, then $|I_2|+|J_2|+ |I_3|+|J_3|\le 3$ provided that $N\ge14$. Now we choose $\varepsilon$ small enough so that $\varepsilon = \dfrac{1}{2CC_1}$. Therefore we obtain better estimates compared to the initial bootstrap assumptions, which shows that $\tau^*=+\infty$.
\subsection{Toy model: quadratic Klein-Gordon equations} We consider the quadratic Klein-Gordon equation
\begin{align}
    (\Box_g-1)\phi = \phi \nabla\phi, \quad (t,x)\in \{(t,x): t\ge2, |x|<t-1\},
\end{align}
with initial data $\phi|_{t=2}=\phi_0$ and $\partial_t\phi|_{t=2}=\phi_1$. Here $\nabla=\nabla_{t,x}$ is the space-time derivative. We make the bootstrap assumptions as in the previous problem. Then we have 
\begin{align*}
    \int_{\tau_0}^\tau \|\partial^I L^J \phi\partial\phi\|_{L^2(\Sigma_{\tau'})}\,d\tau' & \le \sum_{\substack{I_1,J_1 \\ I_2,J_2}}\sum_{j=1}^3\int_{\tau_0}^{\tau} \|\partial^{I_1}L^{J_1}\phi \cdot \partial^{I_2}L^{J_2}\partial_j\phi\|_{L^2(\Sigma_{\tau'})}\,d\tau' \\
    &\qquad+ \sum_{\substack{I_1,J_1 \\ I_2,J_2}}\int_{\tau_0}^{\tau} \|\partial^{I_1}L^{J_1}\phi \cdot \partial^{I_2}L^{J_2}\partial_t\phi\|_{L^2(\Sigma_{\tau'})}\,d\tau'.
\end{align*}
From now one we may assume that $|I_1|+|J_1|\le |I_2|+|J_2|$.
We first consider the second integral
\begin{align*}
     \sum_{\substack{I_1,J_1 \\ I_2,J_2}}\int_{\tau_0}^{\tau} \|\partial^{I_1}L^{J_1}\phi \cdot \partial^{I_2}L^{J_2}\partial_t\phi\|_{L^2(\Sigma_{\tau'})}\,d\tau' & \le C_1 \int_{\tau_0}^\tau \frac t\tau \|\partial^{I_1}L^{J_2}\phi\|_{L^\infty(\Sigma_{\tau'})}\frac\tau t \|\partial_t\partial^{I_2}L^{J_2}\phi\|_{L^2(\Sigma_{\tau'}}\,d\tau' \\
     & \le C_1 \int_{\tau_0}^\tau \frac t\tau t^{-\frac32}C\varepsilon (\tau')^{(|I'_1|+|J_1|+2)\delta}C\varepsilon (\tau')^{\frac12+(|I_2'|+|J_2|)\delta}\,d\tau' \\
     & \le C_1 C^2\varepsilon^2  \tau^{(N+2)\delta},
\end{align*}
and
\begin{align*}
   & \sum_{\substack{I_1,J_1 \\ I_2,J_2}}\sum_{j=1}^3\int_{\tau_0}^{\tau} \|\partial^{I_1}L^{J_1}\phi \cdot \partial^{I_2}L^{J_2}\partial_j\phi\|_{L^2(\Sigma_{\tau'})}\,d\tau'\\
     & \le 3C_1 \int_{\tau_0}^\tau \|\partial^{I_1}L^{J_1}\phi\|_{L^\infty(\Sigma_{\tau'})}\|\partial^{I_2}L^{J_2}\left( \partial_j+\frac{x^j}{t}\partial_t\right)\phi\|_{L^2(\Sigma_{\tau'})}\,d\tau' \\
    & \qquad + 3C_1 \int_{\tau_0}^\tau \|\partial^{I_1}L^{J_1}\phi\|_{L^\infty(\Sigma_{\tau'})}\|\partial^{I_2}L^{J_2}\partial_t\phi\|_{L^2(\Sigma_{\tau'})}\,d\tau'.
\end{align*}
We already have a suitable bound for the second integral. For the first integral, we use the fact $\|(\partial_j+\frac{x^j}{t}\partial_t)\phi\|_{L^2(\Sigma_\tau)}\le E[\phi](\tau)^\frac12$ to get
\begin{align*}
   & 3C_1 \int_{\tau_0}^\tau \|\partial^{I_1}L^{J_1}\phi\|_{L^\infty(\Sigma_{\tau'})}\|\partial^{I_2}L^{J_2}\left( \partial_j+\frac{x^j}{t}\partial_t\right)\phi\|_{L^2(\Sigma_{\tau'})}\,d\tau' \\
     & \le 3C_1\int_{\tau_0}^\tau t^{-\frac32}C\varepsilon \tau^{(|I_1'|+|J_1|+2)\delta}C\varepsilon \tau^{\frac12+(|I_2'|+|J_2|)\delta}\,d\tau \\
    & \le 3C_1 C^2\varepsilon^2\tau^{(N+2)\delta}.
\end{align*}
By choosing $\varepsilon$ small enough we obtain better estimates compared to the initial bootstrap assumptions, so that $\tau^*=+\infty$. 
\section*{Acknowledgements}
Funded by the Deutsche Forschungsgemeinschaft (DFG, German Research Foundation) -- IRTG 2235 -- Project-ID 282638148.
The author thanks Sebastian Herr for helpful discussions, and Gustav Holzegel for valuable feedback on an earlier version of the manuscript.


\end{document}